\documentclass[hidelinks,onefignum,onetabnum]{siamart220329}

\usepackage{bm}
\usepackage{amssymb}
\usepackage{amsmath}
\usepackage{dsfont}
\usepackage{amsfonts}
\usepackage{mathrsfs}
\usepackage{graphicx}
\usepackage{subfig}
\usepackage{subcaption}
\usepackage{multirow}
\usepackage{multicol}

\usepackage{hyperref}
\hypersetup{
    colorlinks=true,
    linkcolor=blue,   
    urlcolor=blue     
}

\ifpdf
\DeclareGraphicsExtensions{.eps,.pdf,.png,.jpg}
\else
\DeclareGraphicsExtensions{.eps}
\fi

\def\XXint#1#2#3{{\setbox0=\hbox{$#1{#2#3}{\int}$ }
		\vcenter{\hbox{$#2#3$ }}\kern-.6\wd0}}

\newsiamremark{remark}{Remark}
\newsiamthm{alg}{Algorithm}
\numberwithin{figure}{section} 
\numberwithin{table}{section}
\headers{SAFE Solver \& LFA}{Shuonan Wu and Jindong Wang}

\title{A robust solver for $\bm{H}(\rm \lowercase{curl})$ convection-diffusion and its local Fourier analysis\thanks{Submitted to the editors DATE.
		\funding{The work of Shuonan Wu is supported in part by the Beijing Natural Science Foundation No. 1232007 and the National Natural Science Foundation of China grant No. 12222101.}}}

\author{Jindong Wang \thanks{School of Mathematical Sciences, Peking University, Beijing 100871, China
		(\email{jdwang@pku.edu.cn}).}
	\and Shuonan Wu\thanks{School of Mathematical Sciences, Peking University, Beijing 100871, China
		(\email{snwu@math.pku.edu.cn}, \url{https://www.math.pku.edu.cn/teachers/snwu/index.html}).}}

\begin{document}

\maketitle

\begin{abstract}
In this paper, we present a robust and efficient multigrid solver based on an exponential-fitting discretization for 2D $\bm{H}$(curl) convection-diffusion problems. By leveraging an exponential identity, we characterize the kernel of $\bm{H}(\rm curl)$ convection-diffusion problems and design a suitable hybrid smoother. This smoother incorporates a lexicographic Gauss-Seidel smoother within a downwind type and smoothing over an auxiliary problem, corresponding to $H(\rm grad)$ convection-diffusion problems for kernel correction. We analyze the convergence properties of the smoothers and the two-level method using local Fourier analysis (LFA). The performance of the algorithms demonstrates robustness in both convection-dominated and diffusion-dominated cases.
\end{abstract}



\section{Introduction}\label{sec:intro}

In this paper, we consider a robust solver based on simplex-averaged finite element (SAFE) method \cite{wu2020simplex} for the $\bm{H}({\rm curl})$ convection-diffusion problem
\begin{align}
{\rm curl} (\varepsilon {\rm curl}\bm{u}+\bm{\beta}\times\bm{u})+\gamma \bm{u} &= \bm{f} ~\text{ in } \Omega,\label{eq:curlcd}\\
\bm{n}\times \bm{u}&=\bm{0}~\text{ on }\partial \Omega\label{eq:curlbdy},
\end{align}
where $\Omega\subset\mathbb{R}^2$ is a polygonal domain, $\varepsilon>0$ is diffusion coefficient, 
$\bm{\beta}\in \bm{W}^{1,\infty}(\Omega)$ denotes the convection field, 
$\gamma$ is a positive constant and $\bm{n}$ is the unit outward normal to $\partial \Omega$. The ${\rm curl}$ operators play different roles at different occurrences. More specifically, the curl is defined as
${\rm curl} \bm{u} := \frac{\partial u_2}{\partial x} - \frac{\partial u_1}{\partial y}$ on the vector-valued function  and ${\rm curl} u := \nabla^\perp u$  on the scalar function, and the cross product is  expressed as $\bm{\beta}\times \bm{u} := \bm{\beta}\cdot \bm{u}^\perp$. Here, $(x,y)^\perp := (y, -x)$.

This problem arises in many applications, typically in  magnetohydrodynamics (MHD) \cite{gerbeau2006mathematical} with the external velocity, where 
the convection coming from the Ohm's law. 
The variational form derived from equations \eqref{eq:curlcd}-\eqref{eq:curlbdy} can be expressed as follows: Find $\bm{u}\in \bm{H}_0(\rm curl,\Omega)$ such that
\begin{equation} \label{eq:variational}
a(\bm{u}, \bm{v}) := (\varepsilon {\rm curl}\bm{u} + \bm{\beta}\times\bm{u}, {\rm curl}\bm{v}) + (\gamma \bm{u}, \bm{v}) = (\bm{f},\bm{v})\quad \forall \bm{v}\in \bm{H}_0(\rm curl,\Omega). 
\end{equation}
One of the main numerical challenges for this problem is the numerical stability with respect to $\varepsilon$ for the convection-dominated case, i.e., $\varepsilon\ll |\bm{\beta}|$. Such an issue has been well-studied for the closely related scalar convection-diffusion problem, presented in conservation form as:
\begin{equation}\label{eq:gradcd}
	-\nabla \cdot (\varepsilon \nabla u+\bm{\beta}u)+\gamma u = f, \quad {\rm in }~ \Omega.
\end{equation}
In this scalar case, standard finite element methods often encounter severe numerical oscillations and instability, particularly due to the presence of interior or boundary layers. Various stabilization techniques have been developed to address this challenge, with many for quasi-uniform meshes, including streamline upwind Petrov Galerkin (SUPG) method \cite{hughes1979multidimentional,brooks1982streamline}, Galerkin least squares finite element method \cite{hughes1989new,franca1992stabilized}, bubble function stabilization \cite{baiocchi1993virtual,brezzi1994choosing,brezzi1998applications}, local projection stabilization (LPS) \cite{becker2001finite,braack2006local,matthies2007unified}, continuous interior penalty (CIP) \cite{burman2004edge,burman2005unified,burman2006continuous} and exponential-fitting method \cite{o1991globally,dorfler1999uniform,sacco1998finite,sacco1999nonconforming,wang1997novel,xu1999monotone,brezzi1989two}.

Noting that the above results are given in $H(\rm grad)$ space. Designing an effective solver for convection-diffusion problems poses a significant challenge due to the inherent asymmetry and indefiniteness of these problems. Various specialized smoothers and multilevel methods have been developed to address these challenges. For instance, a crosswind block iterative method for two-dimensional problems was proposed in \cite{wang1999crosswind}, which divides the problem into crosswind blocks and updates them along the downwind direction to serve as an efficient smoother. A robust smoothing strategy for both two-dimensional and three-dimensional convection-diffusion problems was introduced in \cite{bey1997downwind}. This strategy involves an ordering technique for the grid points known as {\it downwind numbering}, which follows the flow direction and leads to robust multigrid convergence. Regarding the design of efficient solver, a practical and robust multigrid method for convection-diffusion problems was presented in  \cite{kim2003multigrid}. This method is based on a new coarsening technique derived from the graph corresponding to the stiffness matrix, which preserves the properties of the $M$-matrix. A V-cycle iteration using crosswind-block reordering of the unknowns  is proposed in \cite{kim2004uniformly}, which exhibits uniform convergence without any constraint on coarse grids. 

For vector convection-diffusion problems, various stabilization methods have been extensively explored to enhance the stability and accuracy of numerical solutions \cite{heumann2013stabilized, heumann2015stabilized,wang2022discontinuous,wang2024hybridizable}. Recent research has introduced a unified framework for constructing exponentially-fitted basis functions \cite{wang2023exponentially}. Additionally, the development of the Simplex-Averaged Finite Element (SAFE) method \cite{wu2020simplex}, which incorporates operator exponential-fitting techniques, addresses both scalar and vector cases of convection-diffusion problems, advancing stabilization capabilities without introducing additional parameters. A related work on the SAFE method, based on mimetic finite-difference (MFD), is presented in \cite{adler2023stable}.

When designing solvers in the $\bm{H}(\rm curl)$ space, one major challenge is the presence of a larger kernel space, which complicates the development of effective smoothers. There is extensive literature on solvers for Maxwell's equations. Hiptmair \cite{hiptmair1998multigrid} introduced a multigrid method incorporating a Poisson problem as an auxiliary space correction to derive a hybrid smoother. A corresponding two-dimensional local Fourier analysis (LFA) for this multigrid method is provided in \cite{boonen2008local}. Arnold et al. \cite{arnold2000multigrid} demonstrated that with appropriate finite element spaces and suitable additive or multiplicative Schwarz smoothers, the multigrid V-cycle can serve as an efficient solver and preconditioner. Domain decomposition preconditioners for Maxwell's equations are discussed in \cite{pasciak2002overlapping}. Algebraic multigrid methods for solving eddy current approximations to Maxwell's equations are proposed in \cite{bochev2003improved, hu2006toward}. Furthermore, leveraging auxiliary space preconditioning, the Hiptmair-Xu (HX) preconditioner \cite{hiptmair2007nodal} has proven effective in $\bm{H}(\rm curl)$ elliptic problems and is further discussed in \cite{kolev2009parallel}.

For $\bm{H}(\rm curl)$ convection-diffusion problems, there still exists a large kernel space, as well as the asymmetry and indefiniteness of the problem, which also poses challenges. Therefore, there is an urgent need to design smoothers that can effectively capture the kernel  while also utilizing the information of convection field. In this paper, we aim to integrate efforts to address this issue and our contribution is listed as follows:

(1) On structured grids, we borrow idea of  the SAFE scheme to give  a robust discretization and clearly define the downwind direction of vector problems. This allows us to develop a  downwind smoother, thereby integrating convection field information into the iteration.

(2) We aim to precisely characterize the kernel space of $\bm{H}(\rm curl)$  convection-diffusion problems and design an appropriate hybrid smoother to facilitate the development of an efficient geometry multigrid algorithm.

(3) We mainly utilize the local Fourier analysis (LFA) to analyze the properties of hybrid smoother and two-level multigrid for two-dimensional $\bm{H}(\rm curl)$  convection-diffusion problems. The LFA has been introduced for multigrid analysis by Brandt \cite{brandt1977multi}. This provides quantitative insights into the effectiveness and robustness of the hybrid smoother and multigrid algorithm we have developed.

The rest of paper is organized as follows. The SAFE discretization on structured grids, along with its matrix form is described in Section \ref{sec:preliminary}. Section \ref{sec:solver} introduces the corresponding downwind smoother, hybrid smoother, and multigrid algorithm. In Section \ref{sec:lfa}, the LFA of the proposed algorithm is presented. Some numerical results to demonstrate the efficiency of the algorithm are given in Section \ref{sec:numerical}. Finally some concluding remarks are given in Section \ref{sec:conclude}.

\section{Simplex-averaged finite element (SAFE) on rectangle grids}\label{sec:preliminary}
In this section, we extend the idea of the SAFE scheme \cite{wu2020simplex} to rectangular grids, followed by an exploration of the structure of the discrete system. These discussions on matrix or stencil structures not only offer a more intuitive representation of the SAFE algorithm but also lay the groundwork for the subsequent local Fourier analysis (LFA). 

\subsection{SAFE in variational form}
The rectangular grids considered here are denoted by $\mathcal{T}_h$. For the sake of LFA analysis, we limit our focus to square grids where each element has a side length of $h$. Following the principles outlined in SAFE \cite{wu2020simplex}, stable discrete schemes comprise two key ingredients: standard finite elements forming a discrete exact sequence, and identities of exponentially-fitted type. Below, we will delve into each of these components individually.

\paragraph{Discrete de Rham complex on structured grids} We utilize the notation from the monograph \cite{boffi2013mixed}. For a rectangular element $T$, we define 
$P_{k_1,k_2}(T) := \{p: p = \sum_{i\leq k_1, j\leq k_2} a_{ij} x_1^i x_2^j \}$.
The finite element spaces on rectangle grids considered in this paper also facilitates a discrete de Rham complex, characterized by
\begin{align}
\quad 
\mathcal{L}_{[1]}^1 \xrightarrow{\text { grad }} &\mathcal{N}_{[0]} \xrightarrow{\text { curl }} \mathcal{L}_{[0]}^0, \label{eq:2D-deRham}
\end{align}
where the finite element spaces are defined as 
$$ 
\begin{aligned}
\mathcal{L}_{[1]}^1 &:= \{v \in H^1(\Omega):~ v|_T \in \mathcal{L}_{[1]}^1(T) := P_{1,1}(T), \forall T \in \mathcal{T}_h\}, \\
\mathcal{N}_{[0]} &:= \{\bm{v} \in \bm{H}({\rm curl};\Omega):~ \bm{v}|_T \in \mathcal{N}_{[0]}(T) := P_{0,1}(T) \times P_{1,0}(T), \forall T \in \mathcal{T}_h\}, \\
\mathcal{L}_{[0]}^0 &:= \{v \in L^2(\Omega):~ v|_T \in \mathcal{L}_{[0]}^0(T) :=P_{0,0}(T), \forall T \in \mathcal{T}_h\}. 
\end{aligned}
$$ 
We note that the degrees of freedom associated with vertices, edges, and elements clearly reflect geometric properties (see Figure \ref{fig:curl-DOFs} for $\mathcal{N}_{[0]}(T)$). The canonical interpolation can be induced by the degrees of freedom and is denoted by $\Pi_h^{\rm grad}$, $\Pi_h^{\rm curl}$, and $\Pi_h^{\rm 0}$.

\paragraph{Identities of exponentially-fitted type} Firstly, we define the following fluxes:
 \begin{align} 
 	\mathcal{J}^{\rm grad}_{\varepsilon, \bm{\beta}} {u} &:= \varepsilon{\rm grad} {u}+\bm{\beta}{u} 
	= \varepsilon({\rm grad} u+ \varepsilon^{-1}\bm{\beta}{u}),  \label{eq:Jgrad} \\
 	\mathcal{J}^{\rm curl}_{\varepsilon, \bm{\beta}} \bm{u}&:= \varepsilon{\rm curl } \bm{u}+\bm{\beta}\times\bm{u} = \varepsilon({\rm curl}\bm{u}+  \varepsilon^{-1}\bm{\beta}\times\bm{u}), \label{eq:Jcurl}
 \end{align}
where $\mathcal{J}^{\rm curl}_{\varepsilon, \bm{\beta}}$ is exactly the operator acting on $\bm{u}$ in the variational form \eqref{eq:variational}. In \cite[Lemma 3.1]{wu2020simplex}, a unified identity is provided for several types of flux operators, enabling the design of exponentially-fitted finite element methods. This paper primarily focuses on the discussion of $\bm{H}({\rm curl})$ problems, thus we only list two particular cases of this unified identity, namely when $\bm{\beta}/\varepsilon$ is constant:
\begin{equation} \label{eq:Jcurl-exp}
\mathcal{J}_{\varepsilon, \bm{\beta}}^{\rm grad}{u} = \varepsilon E_{\bm{\beta}/\varepsilon}^{-1}{\rm grad} E_{\bm{\beta}/\varepsilon}{u}, \qquad
	\mathcal{J}_{\varepsilon, \bm{\beta}}^{\rm curl}\bm{u}=\varepsilon E_{\bm{\beta}/\varepsilon}^{-1} {\rm curl} E_{\bm{\beta}/\varepsilon}\bm{u}.
\end{equation}
Here, for any vector field $\bm{\delta}$, we denote $E_{\bm{\delta}} (\bm{x}) := \exp(\bm{\delta}\cdot\bm{x})$. This implies that flux operators can be obtained via exponential transformation and its inverse, combined with standard differential operators.

Below, we present the discrete flux operator for $\bm{H}({\rm curl})$ problems. Firstly, we utilize a piecewise constant approximation for $\varepsilon$ and $\bm{\beta}$, denoted as $\bar{\varepsilon}$ and $\bar{\bm{\beta}}$ respectively. Then, we formally apply the local commutative property between the canonical interpolation and the differential operator. For instance, ${\rm curl} \Pi_T^{\rm curl} = \Pi_T^{0} {\rm curl}$, leading to the expression:
$$
(\Pi_T^{0} E_{\bar{\bm{\beta}}/\bar{\varepsilon}}) \bar{\varepsilon}^{-1} \mathcal{J}^{\rm curl}_{\bar{\varepsilon}, \bar{\bm{\beta}}} {\bm{u}} 
= \Pi_T^{0} {\rm curl} (E_{\bar{\bm{\beta}}/\bar{\varepsilon}} \bm{u})= {\rm curl}\circ \Pi_T^{\rm curl}  (E_{\bar{\bm{\beta}}/\bar{\varepsilon}} \bm{u}).
$$
This yields the following discrete flux operator.
\begin{definition}[discrete flux operator] \label{df:discreteflux}
	The local discrete flux operator $\mathcal{J}_{\bar{\varepsilon},\bar{\bm{\beta}},T}^{\rm curl}$ is defined by
	\begin{equation}\label{eq:discreteflux}
		\mathcal{J}^{\rm curl}_{\bar{\varepsilon}, \bar{\bm{\beta}},T} \bm{v}_h := \bar{\varepsilon} (\Pi_T^{0} E_{\bar{\bm{\beta}}/\bar{\varepsilon}})^{-1} \circ {\rm curl} \circ \Pi_T^{\rm curl} (E_{\bar{\bm{\beta}}/\bar{\varepsilon}} \bm{v}_h) \quad \forall \bm{v}_h\in \mathcal{N}_{[0]}(T).
	\end{equation}
\end{definition}
\begin{remark}[well-posedness of discrete flux operator]
In the design of the SAFE scheme, a critical issue is the well-posedness of the discrete inverse exponential transformation $(\Pi_T^{0} E_{\bar{\varepsilon}^{-1}\bar{\bm{\beta}}})^{-1}$ in the discrete flux operator \eqref{eq:discreteflux}. This is relatively straightforward for lowest-order elements \cite{wu2020simplex}, while discussions on higher-order elements can be found in \cite{wu2020unisolvence}. Another significant concern is stable implementation, particularly as the diffusion coefficient $\varepsilon \to 0^+$. In the next section, we will clarify these points by providing specific expressions for the discrete system.
\end{remark}

For the discrete space $V_h := \mathcal{N}_{[0]} \cap \bm{H}_0({\rm curl};\Omega) := \{\bm{v}_h \in \mathcal{N}_{[0]}: \bm{v}_h \times \bm{n} = 0 \text{ on }\partial\Omega \}$, directly applying the discrete flux operator \eqref{eq:discreteflux} yields the variational form of the SAFE scheme: Find $\bm{u}_h \in V_h$ such that 
$$
a_h(\bm{u}_h, \bm{v}_h) = (\bm{f}, \bm{v}_h) \qquad \forall \bm{v}_h \in V_h,
$$ 
where
\begin{equation} \label{eq:SAFE}
a_h(\bm{u}_h,\bm{v}_h):= \sum_{T \in \mathcal{T}_h}(\mathcal{J}^{\rm curl}_{\bar{\varepsilon}, \bar{\bm{\beta}},T} \bm{u}_h, {\rm curl}\bm{v}_h)_T + (\gamma \bm{u}_h, \bm{v}_h) \quad \bm{u}_h,\bm{v}_h\in V_h.
\end{equation}

\begin{remark}[SAFE scheme on 3D cubic grid]
Similarly, for the 3D problem on a cubic grid, we can utilize the following 3D de Rham complex:
\begin{align}
	\mathcal{L}_{[1]}^1 \xrightarrow{\text { grad }} &\mathcal{N}_{[0]} \xrightarrow{\text { curl }} \mathcal{RT}_{[0]} \xrightarrow{\text { div }} \mathcal{L}_{[0]}^0. \label{eq:3D-deRham}
\end{align}
The definition of the corresponding discrete space can be found in \cite{boffi2013mixed}. Following this complex, we can fully replicate the above process to define the corresponding discrete flux operators.  This enables us to derive the SAFE scheme for 3D $\bm{H}$(curl) convection-diffusion problems on cubic grid. 
\end{remark}

\subsection{Local stiffness matrix}
We provide the local stiffness matrix in \eqref{eq:SAFE}, and demonstrate the well-posedness of the discrete flux in SAFE through this derivation. We are working within the element \(T = [0,h] \times [0,h]\), with its four edges $e_i$ ($i=1,2,3,4$). The degrees of freedom and basis functions $\bm{\phi}_i$ of its \(\bm{H}(\text{curl})\) space are illustrated in Figure \ref{fig:curl-DOFs}. 
\begin{figure}
 \centering
 \begin{minipage}{.4\textwidth}
  \includegraphics[width=1\textwidth]{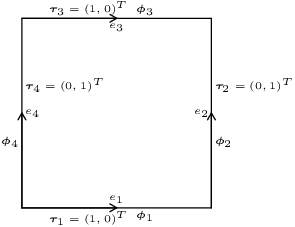}
 \end{minipage}
 \begin{minipage}{.4\textwidth}
  $${\rm DOFs}:\ l_i(\bm{v}):=\int_{e_i} \bm{v}\cdot \bm{\tau}_i$$
   $$\begin{aligned}
  \bm{\phi}_1&=\begin{pmatrix}
   {h-x_2\over h^2}\\0
  \end{pmatrix},~\bm{\phi}_2=\begin{pmatrix}
 0\\ {x_1\over h^2}
  \end{pmatrix},\\
  \bm{\phi}_3&=\begin{pmatrix}
   {x_2\over h^2}\\0
  \end{pmatrix},~\bm{\phi}_4=\begin{pmatrix}
  0\\{h-x_1\over h^2}
  \end{pmatrix}. \end{aligned}
  $$
 \end{minipage}
 \caption{DOFs and basis of $\mathcal{N}_{[0]}$ on $T = (0,h)^2$.} \label{fig:curl-DOFs}
\end{figure}

In 2D discrete de Rham complex \eqref{eq:2D-deRham}, the subsequent space of $\mathcal{N}_{[0]}$ is \(\mathcal{L}_{[0]}^0\), also known as the piecewise constant space. Thus, the discrete inverse exponential transformation in \eqref{eq:discreteflux} has the explicit form:
$$
(\Pi_T^0 E_{\bar{\bm{\beta}}/\bar{\varepsilon}})^{-1} = h^2 \left( \int_0^h \mathrm{e}^{\bar{\beta}_1 x_1 /\bar{\varepsilon} } \mathrm{d}x_1 \right)^{-1}
 \left( \int_0^h \mathrm{e}^{\bar{\beta}_2 x_2 /\bar{\varepsilon} } \mathrm{d}x_2 \right)^{-1}.
$$ 

For any $\varepsilon > 0$, we define the (1D) Bernoulli function:
\begin{equation}\label{eq:bernoulli}
	B_\varepsilon(s) := \varepsilon \frac{1}{\int_0^1 e^{sx/\varepsilon}{\rm d}x}.
\end{equation}
Then, by denoting $\bar{b}_i := \bar{\beta}_i h, i=1,2$, a direct calculation shows 
$$
\begin{aligned}
\mathcal{J}^{\rm curl}_{\bar{\varepsilon}, \bar{\bm{\beta}},T} \bm{\phi}_1 
&= h^{-1} \bar{\varepsilon} \left( \int_0^h \mathrm{e}^{\bar{\beta}_2 x_2 /\bar{\varepsilon} } \mathrm{d}x_2 \right)^{-1}
= h^{-2} \bar{\varepsilon} \left( \int_0^1 \mathrm{e}^{\bar{b}_2 x_2 /\bar{\varepsilon} } \mathrm{d}x_2 \right)^{-1} = h^{-2}B_{\bar{\varepsilon}}(\bar{b}_2), \\
\mathcal{J}^{\rm curl}_{\bar{\varepsilon}, \bar{\bm{\beta}},T} \bm{\phi}_3
&= - h^{-1} \left( \int_0^h \mathrm{e}^{\bar{\beta}_2 x_2 /\bar{\varepsilon} } \mathrm{d}x_2 \right)^{-1} \mathrm{e}^{\bar{\beta}_2 h/\bar{\varepsilon}}
= -h^{-2}B_{\bar{\varepsilon}}(-\bar{b}_2),\\
\mathcal{J}^{\rm curl}_{\bar{\varepsilon}, \bar{\bm{\beta}},T} \bm{\phi}_2 &= h^{-2}B_{\bar{\varepsilon}}(-\bar{b}_1), \qquad 
\mathcal{J}^{\rm curl}_{\bar{\varepsilon}, \bar{\bm{\beta}},T} \bm{\phi}_4 = -h^{-2}B_{\bar{\varepsilon}}(\bar{b}_1).
\end{aligned} 
$$
Therefore, the local stiffness matrix is 
\begin{equation} \label{eq:local-stiff}
\left[(\mathcal{J}^{\rm curl}_{\bar{\varepsilon}, \bar{\bm{\beta}},T} \bm{\phi}_i, {\rm curl}\bm{\phi}_j)\right]_{ij}
= \frac{1}{h^2}
\begin{bmatrix}
B_{\bar{\varepsilon}}(\bar{b}_2) & B_{\bar{\varepsilon}}(\bar{b}_2) & -B_{\bar{\varepsilon}}(\bar{b}_2) & -B_{\bar{\varepsilon}}(\bar{b}_2) \\
B_{\bar{\varepsilon}}(-\bar{b}_1) & B_{\bar{\varepsilon}}(-\bar{b}_1) & -B_{\bar{\varepsilon}}(-\bar{b}_1) & -B_{\bar{\varepsilon}}(-\bar{b}_1) \\
-B_{\bar{\varepsilon}}(-\bar{b}_2) & -B_{\bar{\varepsilon}}(-\bar{b}_2) & B_{\bar{\varepsilon}}(-\bar{b}_2) & B_{\bar{\varepsilon}}(-\bar{b}_2) \\
-B_{\bar{\varepsilon}}(\bar{b}_1) & -B_{\bar{\varepsilon}}(\bar{b}_1) & B_{\bar{\varepsilon}}(\bar{b}_1) & B_{\bar{\varepsilon}}(\bar{b}_1) \\
\end{bmatrix}.
\end{equation}

\subsection{Grid operator form and stencil representation}
This subsection aims to characterize the discrete system using stencil terminology. Our objective is to depict the geometric properties (vertices, edges, faces) of structured grids using indices, a task that is particularly straightforward for structured grids. For instance, the grid points of rectangular grids can be mapped to two-dimensional integers as $\bm{x} = h\bm{k} + (x_1^0, x_2^0) \mapsto \bm{k} \in \mathbb{Z}^2$, for any grid point $\bm{x}$. Without loss of generality, we can designate the point $(0,0)$ as a grid point. In this case, $(x_1^0, x_2^0)$ can be chosen as $(0,0)$. 

At this point, a node $\bm{x}=(k_1,k_2)h$ is referred simply to $(k_1,k_2)$. Edges and elements will be referred using their center points. For instance, $({r_1+s_1\over2},{r_2+s_2\over2})$ refers to the edge connecting the nodes $(r_1,r_2)$ and $(s_1,s_2)$, and $({p_1+q_1+r_1+s_1\over 4},{p_2+q_2+r_2+s_2\over 4})$ refers to the element formed by nodes $(p_1,p_2),(q_1,q_2),(r_1,r_2)$, and $(s_1,s_2)$. 

\paragraph{Index sets} When characterizing discrete problems on rectangular grids using stencil terminology, it is convenient to extend the grid $\mathcal{T}_h$ to infinity. For simplicity, we introduce a displaced index set $I(\Delta_1,\Delta_2)$, as utilized in \cite{boonen2008local}. Here, with $\Delta_1,\Delta_2\in\mathbb{R}$, this set is defined as
$$
I(\Delta_1,\Delta_2):=\{(k_1+\Delta_1,k_2+\Delta_2):k_1,k_2\in\mathbb{Z}\}.
$$
This allows us to define the sets of indices for all nodes, horizontal and vertical edges, and faces respectively as follows: 
\begin{equation} \label{eq:index-set}
N:=I(0,0), \quad E_1:=I({1\over 2},0), \quad E_2:=I(0,{1\over 2}), \quad F:=I({1\over 2},{1\over 2}).
\end{equation}
We also denote $E = E_1 \cup E_2$ as the index set of edges. 

\paragraph{Grid functions} By associating geometric quantities with indices, grid functions defined on them can also be acquired by selecting appropriate indices. Following the notation of \cite{boonen2008local}, we employ $\mathcal{N}$, $\mathcal{E}$, and $\mathcal{F}$ to signify grid function spaces defined on $N$, $E$, and $F$, respectively. 

For a function $w$ defined on $\mathbb{R}^2$, we can naturally interpolate it into grid functions based on the indices. Specifically, for an index set $S$ (where $S$ can be the node index set $N$, the edge index set $E$, or the face index set $F$), we define the grid function
\begin{equation} \label{eq:grid-function}
w_{S}(\bm{k}) := w(h\bm{k})  \quad\forall \bm{k} \in S, \quad S=N,E,F.
\end{equation}
Further, for a grid function $f$, we denote $f_{\bm{k}}:=f(\bm{k})$, provided it is well-defined.

\paragraph{Grid operators} 
We refer to a grid operator as a linear mapping on grid function spaces. Specifically, $L: \mathcal{E} \to \mathcal{E}$ is called an edge operator, which has the following stencil representation:
\begin{equation}\label{eq:stencil}
	(L f)_{\bm{k}} := \left\{
	\begin{aligned}
		\sum_{\bm{t} \in X} l_{\bm{t}}^{[1]} f_{\bm{k} + \bm{t}} \quad & \text{if } \bm{k} \in E_1, \\
		\sum_{\bm{t} \in X} l_{\bm{t}}^{[2]} f_{\bm{k} + \bm{t}} \quad & \text{if } \bm{k} \in E_2,
	\end{aligned} \right. \qquad \forall f \in \mathcal{E}, \quad \text{with }X := I(0,0) \cup I(\frac{1}{2}, \frac{1}{2}).
\end{equation}
Here, $l_{\bm{t}}^{[1]}$ and $l_{\bm{t}}^{[2]}$ denote the coefficients of the horizontal and vertical stencils, respectively.

\paragraph{Stencil representation of SAFE scheme}
Note that the degrees of freedom for $\mathcal{N}_{[0]}$ are the tangential component to the edges. Therefore, the bilinear form $a_h(\cdot, \cdot)$ in the SAFE scheme \eqref{eq:SAFE} can actually be considered as an edge operator that enforces boundary conditions. For the stiffness and mass components, the operator form of $a_h(\cdot, \cdot)$ can be written as $A := K + \gamma M$. Using the local stiffness matrix \eqref{eq:local-stiff}, we obtain the stencil representations for $K$ and $M$ as follows:
\begin{subequations} \label{eq:stencil-SAFE}
\begin{equation} \label{eq:stencil-h}
\begin{aligned}
	K\big|_{E_1}: & {1\over h^2} \left[\begin{array}{rcr}
		\circ & -B_{\bar{\varepsilon}}(-\bar{b}_2) & \circ \\
		-B_{\bar{\varepsilon}}(\bar{b}_1) & & B_{\bar{\varepsilon}}(-\bar{b}_1) \\
		\bullet & B_{\bar{\varepsilon}}(\bar{b}_2)+B_{\bar{\varepsilon}}(-\bar{b}_2) & \bullet \\
		B_{\bar{\varepsilon}}(\bar{b}_1) & & -B_{\bar{\varepsilon}}(-\bar{b}_1) \\
		\circ & -B_{\bar{\varepsilon}}(\bar{b}_2) & \circ
	\end{array}\right], \\
	 M\big|_{E_1}: & ~~~~ \left[\begin{array}{rcr}
		\circ & {1\over 6} & \circ \\
		0 & & 0 \\
		\bullet & {2\over 3} & \bullet \\
		0& & 0 \\
		\circ & {1\over 6} & \circ
	\end{array}\right],
\end{aligned}
\end{equation}
and
\begin{equation} \label{eq:stencil-v}
\begin{aligned}
	K\big|_{E_2}:& {1\over h^2} \left[\begin{array}{ccccc}
		\circ & -B_{\bar{\varepsilon}}(-\bar{b}_2) & \bullet &  B_{\bar{\varepsilon}}(-\bar{b}_2)  & \circ \\
		-B_{\bar{\varepsilon}}(\bar{b}_1)& & B_{\bar{\varepsilon}}(\bar{b}_1)+B_{\bar{\varepsilon}}(-\bar{b}_1) & & -B_{\bar{\varepsilon}}(-\bar{b}_1) \\
		\circ & B_{\bar{\varepsilon}}(\bar{b}_2) & \bullet & -B_{\bar{\varepsilon}}(\bar{b}_2) & \circ
	\end{array}\right], \\
	M\big|_{E_2}:&~~~~ \left[\begin{array}{ccccc}
		\circ & 0 & \bullet & 0 & \circ \\
		\frac{1}{6} & & \frac{2}{3} & & \frac{1}{6} \\
		\circ & 0 & \bullet & 0 & \circ
	\end{array}\right].
\end{aligned}
\end{equation}
\end{subequations}
Here, solid circles represent the nodes of the edge to which the result of the stencil operator is assigned, while all other nodes are represented by open circles.

\begin{remark}[stencil on 3D cubic grids] For the 3D cubic grids, we can still derive stencil representation as described above. Moreover, the structural property of the grid also leads to the stencil form depending solely on 1D Bernoulli function \eqref{eq:bernoulli}. This contrasts with unstructured grids, where the two-dimensional Bernoulli function is required \cite{wu2020simplex}.
\end{remark}

\section{A robust solver for SAFE}\label{sec:solver}
In this section, we present a fast and stable iterative solver tailored for the SAFE discretization. With the introduction of the convection term, the design of its smoother exhibits two key characteristics: (i) we devise a lexicographic Gauss-Seidel smoother within a downwind type; (ii) recognizing the expansive kernel space of the $\bm{H}({\rm curl})$ convection-diffusion operator, we introduce a corresponding smoothing mechanism. The combination of these two components yields a {hybrid smoother} implemented on a single grid. Subsequently, we employ the classical multigrid method to develop a solver that consistently addresses both convection and diffusion.

\subsection{Downwind lexicographic Gauss-Seidel smoother}\label{subsec:downwind}
Starting from the stencil representation \eqref{eq:stencil-SAFE}, we will discuss the design scheme of the lexicographic Gauss-Seidel smoother. First, we examine the influence of convection velocity on the stencil representation. Here, we first consider the case of constant diffusion and coefficient velocity field, i.e., $\varepsilon$ and $\bm{\beta}=(\beta_1,\beta_2)^T$ are constant, ensuring same stencil across all horizontal (or vertical) edges. Notably, the 1D Bernoulli function \eqref{eq:bernoulli} exhibits the following property:
\begin{equation} \label{eq:bernoulli-property}
\lim_{\varepsilon\rightarrow 0^+} B_\varepsilon(s) =\left\{\begin{aligned}
	&-s &s\le 0,\\
	&0 &s\ge 0.
\end{aligned}\right.
\end{equation}
In other words, in the limiting case, the non-zero value of the stencil in \eqref{eq:stencil-SAFE} depends solely on the sign of each velocity component, specifically the quadrant in which the velocity field resides. Hence, by considering the limiting of diffusion coefficient $\varepsilon$ as the positive part of the stencil of $K$, this establishes a splitting pattern for the lexicographic Gauss-Seidel smoother.

For velocities in the $i$-th quadrant (where $i=1,2,3,4$), we denote the aforementioned splitting for $K$ as $K^{i,+}$ and $K^{i,-}:=K-K^{i,+}$, indicating $K=K^{i,+}+K^{i,-}$. For instance, regarding the splitting of the stencil for the velocity field in the first quadrant, as illustrated in Figure \ref{fig:2dupdate}. Here, bold edges indicate the use of new values in the lexicographic Gauss-Seidel smoothing process, as referenced in \eqref{eq:stencil} for the specific form of the stencil. For the stencil decomposition of $A$, we employ that of $K$, thus yielding the lexicographic Gauss-Seidel iteration for the velocity field in the $i$-th quadrant: 
$$
A^{i,+} w^{[\nu+1]} = f - A^{i,-}w^{[\nu]},
$$
where the superscript $[\nu]$ denoting the iteration index.
 
\begin{figure}[!htbp]
\centering
\includegraphics[width=.7\textwidth]{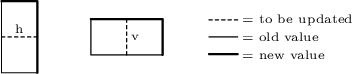}
\caption{Updating order for the (lexicographic) Gauss-Seidel smoother for the convection in the first quadrant ($b_1,b_2>0$).}
\label{fig:2dupdate}
\end{figure}
Figure \ref{fig:2dupdate} illustrates the update order for the splitting $(A^{1,+},A^{1,-})$, whose sweep direction is {\it downwind} considering the convection in the first quadrant. In the conservation form \eqref{eq:curlcd}, where convective terms are within the differential operator, their practical function is to induce convection in the $-\bm{\beta}$ direction, rendering downwind smoother also reasonable. 

In general, the quadrant of velocity field cannot be assumed to be consistent everywhere. It should be noted that the assumption of a constant velocity field is only for determining the corresponding splitting $A^{i,\pm}$. Therefore, in practice, we only need to smooth out the four possible cases of the velocity field in the two-dimensional situation. This results in the downwind smoother proposed in this paper, which can be represented in pseudocode as follows:

\begin{alg}[downwind Gauss-Seidel smoother]\label{alg:downwind}
	
	$\xi \leftarrow S^{\rm dw}(A,\xi,f)$
	
	\{
	
	\text{ for }i=1:4
	
	\qquad $ \xi\leftarrow (A^{i,+})^{-1}(f-A^{i,-}\xi)$
	
	\text{end for}
	
	\}	
\end{alg}

\subsection{Hybrid smoother with kernel correction} \label{subsec:kernel}
It is widely known that the design of solvers for $\bm{H}({\rm curl})$ problems differs fundamentally from scalar problems due to the significant {\it kernel space} of the curl operator. In fact, as indicated in the seminal work by Hiptmair \cite{hiptmair1998multigrid}, for $\bm{H}({\rm curl})$ diffusion problems, restricting the problem to $\mathrm{ker}({\rm curl})$ results in no amplification of high frequencies. This viewpoint remains valid for $\bm{H}({\rm curl})$ convection-diffusion problems, and we will approach this issue from the variational form \eqref{eq:variational} with $\varepsilon$ and $\bm{\beta}$ being constant. By employing the exponentially-fitted identity \eqref{eq:Jcurl-exp}, it becomes evident that
\begin{equation} \label{eq:exactness-J}
\mathcal{J}_{\varepsilon, \bm{\beta}}^{\rm curl} \circ {\mathcal J}_{\varepsilon, \bm{\beta}}^{\rm grad} = \varepsilon^2 E_{\bm{\beta}/\varepsilon}^{-1} {\rm curl} \circ {\rm grad} E_{\bm{\beta}/\varepsilon} = 0.
\end{equation}

Therefore, when certain conditions hold in the domain (c.f. \cite[Theorem 2.9]{girault1986finite}), it can be expressed as follows: $\mathrm{ker}({\rm curl}) = {\rm grad} H_0^1(\Omega)$ and $\mathrm{ker}(\mathcal{J}_{\varepsilon, \bm{\beta}}^{\rm curl}) = {\mathcal J}_{\varepsilon, \bm{\beta}}^{\rm grad} H_0^1(\Omega)$, with these two kernel spaces being isomorphic. Consequently, on the $\mathrm{ker}(\mathcal{\bm{J}}_{\varepsilon,\bm{\beta}}^{\rm curl}) \times \mathrm{ker}({\rm curl})$ space, we have
$$
a(\cdot, \cdot)|_{\mathrm{ker}(\mathcal{J}_{\varepsilon,\bm{\beta}}^{\rm curl}) \times \mathrm{ker}({\rm curl})} \Leftrightarrow 
({\mathcal J}_{\varepsilon, \bm{\beta}}^{\rm grad} \cdot , {\rm grad} \cdot)
$$
with the right-hand side living in $H_0^1(\Omega)$.
An important observation is that the bilinear form concerning the kernel space described above corresponds to the equation of scalar convection-diffusion, with the transfer operators being ${\mathcal J}_{\varepsilon, \bm{\beta}}^{\rm grad}$ and ${\rm grad}^*$ (the dual operator of grad). Consequently, smoothing within the kernel space can be effectively achieved using the proper smoothing operators for scalar convection-diffusion problems.

In the discrete setting, it is necessary to define the corresponding discrete transfer operators between the discrete spaces. For the discrete gradient flux, using the exponentially-fitted identity \eqref{eq:Jcurl-exp}, it can be defined in the same manner as the discrete flux operator in Definition \ref{df:discreteflux}:
\begin{equation}\label{eq:discreteflux-grad}
    \mathcal{J}^{\rm grad}_{\bar{\varepsilon}, \bar{\bm{\beta}},T} v_h := \bar{\varepsilon} (\Pi_T^{\rm curl} E_{\bar{\bm{\beta}}/\bar{\varepsilon}})^{-1} \circ {\rm grad} \circ \Pi_T^{\rm grad} (E_{\bar{\bm{\beta}}/\bar{\varepsilon}} v_h) \quad \forall v_h\in \mathcal{L}_{[1]}^1(T).
\end{equation}
Note that the interpolations $\Pi_T^{\rm curl}$ and $\Pi_T^{\rm grad}$ use only the information at points and edges. Therefore, the constants $\bar{\varepsilon}$ and $\bar{\bm{\beta}}$ in \eqref{eq:discreteflux-grad} are {\it approximated on the edges} (here we slightly abuse notation, as the elementwise constants in $\mathcal{J}^{\rm curl}_{\bar{\varepsilon}, \bar{\bm{\beta}}}$ are still denoted by $\bar{\varepsilon}$ and $\bar{\bm{\beta}}$). Since the discrete gradient flux acts as a transfer operator, it is essential that the global operator $\mathcal{J}^{\rm grad}_{\bar{\varepsilon}, \bar{\bm{\beta}}}$ maps $\mathcal{L}_{[1]}^1(\mathcal{T}_h)$ to $\mathcal{N}_{[0]}(\mathcal{T}_h)$, which also necessitates the edge-based constant approximation. As a result, we have the corresponding global grid operator, denoted as $J^{\rm grad}_{\bar{\varepsilon}, \bar{\bm{\beta}}}: \mathcal{N} \to \mathcal{E}$, with the following stencil representation:
$$
(J^{\rm grad}_{\bar{\varepsilon},\bar{\bm{\beta}}} f)_{\bm{k}} =
\begin{cases}
-B_{\bar{\varepsilon}}(\bar{b}_1)f_{k_1-\frac12,k_2} + B_{\bar{\varepsilon}}(-\bar{b}_1) f_{k_1+\frac12,k_2} & \text{if } \bm{k} \in E_1, \\
-B_{\bar{\varepsilon}}(\bar{b}_2)f_{k_1,k_2-\frac12} + B_{\bar{\varepsilon}}(-\bar{b}_2) f_{k_1,k_2+\frac12} & \text{if } \bm{k} \in E_2, 
\end{cases}\quad \forall f \in \mathcal{N}.
$$
For the grad operator, its corresponding grid operator is denoted as $G = J^{\rm grad}_{1,\bm{0}}: \mathcal{N} \to \mathcal{E}$. Ultimately, we derive the following pseudocode for the hybrid smoother, presented in the form of grid operators.

\begin{alg}[hybrid smoother]\label{alg:hybrid}
	
	$\xi \leftarrow S^{\rm hybrid}(A, \xi, f)$
	
 	\{
	
\quad Downwind Gauss-Seidel sweep $\xi \leftarrow S^{\rm dw}(A,\xi, f)$ \qquad [pre-smoothing]
	
\quad $\rho \leftarrow f -A\xi$

\quad $\rho \leftarrow G^T \rho$, $\psi \leftarrow 0$, $A^{\rm aux} \leftarrow G^T A J^{\rm grad}_{\bar{\varepsilon}, \bar{\bm{\beta}}}$

\quad Gauss-Seidel sweep over node-based auxiliary problem $A^{\rm aux} \psi = \rho$

\quad  $\xi \leftarrow \xi + J^{\rm grad}_{\bar{\varepsilon}, \bar{\bm{\beta}}} \psi $	

\quad Downwind Gauss-Seidel sweep  $\xi \leftarrow S^{\rm dw}(A,\xi,f)$  \qquad [post-smoothing]

\}
\end{alg}

In the hybrid smoother detailed above, it is necessary to address the auxiliary problem through a sweeping. Subsequent analysis (Theorem \ref{thm:auxcharacter}) indicates that this auxiliary problem aligns with the SAFE method for $H({\rm grad})$ convection-diffusion problems when $\varepsilon$ and $\bm{\beta}$ are constants. Specifically, its bilinear form is
\begin{equation} \label{eq:SAFE-grad}
a_h^{\rm grad}(u_h, v_h) := (\mathcal{J}^{\rm grad}_{\varepsilon,\bm{\beta}}u_h,{\rm grad}v_h)_T \quad u_h,v_h\in \mathcal{L}_{[1]}^1\cap H_0^1(\Omega).
\end{equation}
The stencil corresponding to this bilinear form involves the 8 surrounding points of a node and is structured as follows:
\begin{equation}\label{eq:gradcdstencil}
{\tiny
\begin{matrix}
\frac{1}{6}
\end{matrix}
\begin{bmatrix}
	-B_{\varepsilon}(b_1)-B_{\varepsilon}(-b_2) & B_{\varepsilon}(b_1)-4B_{\varepsilon}(-b_2)+B_{\varepsilon}(-b_1) &-B_{\varepsilon}(-b_1) - B_{\varepsilon}(-b_2) \\
	B_{\varepsilon}(b_2)-4B_{\varepsilon}(b_1)+B_{\varepsilon}(-b_2) & 4(B_{\varepsilon}(b_1)+B_{\varepsilon}(-b_1)+B_{\varepsilon}(b_2)+B_{\varepsilon}(-b_2)) & B_{\varepsilon}(b_2)-4B_{\varepsilon}(-b_1)+B_{\varepsilon}(-b_2) \\
	-B_{\varepsilon}(b_1) - B_{\varepsilon}(b_2) & B_{\varepsilon}(b_1)-4B_{\varepsilon}(b_2)+B_{\varepsilon}(-b_1) &-B_{\varepsilon}(-b_1) - B_{\varepsilon}(b_2)
\end{bmatrix}
}
\end{equation}
Similarly, by noting the limiting properties of the Bernoulli function \eqref{eq:bernoulli-property}, we can design a downwind smoother. In the numerical experiments, we apply the following updating order for the convection in the first quadrant ($b_1,b_2>0$):
\begin{figure}[!htbp]
	\centering
	\includegraphics[width=.5\textwidth]{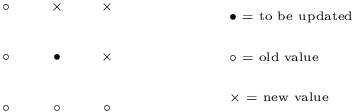}
	\caption{Updating order for the (lexicographic) Gauss-Seidel smoother of auxiliary problem for the convection in the first quadrant ($b_1,b_2>0$).}
	\label{fig:2dnode}
\end{figure}

The rationale behind this design is that the ``new values'' in the upper-right part of the stencil simultaneously include $B_{\varepsilon}(-b_1)$ and $B_{\varepsilon}(-b_2)$, which do not degenerate as $\varepsilon \to 0^+$. Additionally, this selection ensures that the lexicographic Gauss-Seidel smoother completes in a sweeping manner. The updating order of the convection in the other quadrant can be given in a similar way. Following the discussion in Section \ref{subsec:downwind}, in practical implementation, sweeping is still conducted over all four possible quadrants of the velocity field.

\begin{remark}[smoothers for auxiliary problem]
Existing techniques such as the crosswind block smoother \cite{wang1999crosswind}, downwind smoother \cite{bey1997downwind}, and various multigrid methods \cite{kim2003multigrid,kim2004uniformly} can also be utilized effectively for sweeping.
\end{remark}

\subsection{Multigrid method}
On single level, we have designed a hybrid smoother that incorporates convection information and kernel correction. It is well known that efficient solvers can be obtained by designing multilevel algorithms when diffusion dominates, as discussed in \cite{hiptmair1998multigrid}. Therefore, to make the solver widely applicable to various scenarios of convection or diffusion, we adopt the classical geometric multigrid approach to achieve multilevel effectiveness.

Let $\{\mathcal{T}_l\}_{l=0}^L$ be a family of nested structured grid of the computation domain $\Omega$ obtained by successive refinement of any given coarse grid $\mathcal{T}_0$. We denote by $V_l$ the finite element space associated with $\mathcal{N}_{[0]}$ on $\mathcal{T}_l$. It is easy to see that the spaces $\{V_l\}_{l=0}^L$ are nested, i.e., $V_0\subset V_1\subset\cdots\subset V_L$.

We can define  the sesquilinear forms $\{a_l(\cdots,\cdot)\}_{l=0}^L$ on $V_l$ defined as \eqref{eq:SAFE} using the discrete flux operators on $V_l$, which are nonnested since the discrete flux operator is not nested. For $0\le l\le L$, we define the operator $A_l:V_l\rightarrow V_l$ according to
$$
(A_l \bm{w},\bm{v}):= a_l(\bm{w},\bm{v})\quad \forall \bm{w},\bm{v}\in V_l.
$$
For the nested spaces $V_l$, it is straightforward to define the canonical N\'ed\'elec restriction and prolongation operators  $R_l^{l-1}: V_l\rightarrow V_{l-1}$ and $P_{l-1}^l:V_{l-1}\rightarrow V_l$ between $V_{l-1}$ and $V_l$, induced by the natural embedding of these spaces (see \cite{hackbusch1986theorie}).

Using the above operators and the hybrid smoother in Algorithm \ref{alg:hybrid}, we can now outline the multigrid method for solving the SAFE method on level $L$, corresponding to the V(1,1)-cycle multigrid method in the operator form. 
\begin{alg}[Multigrid V(1,1)-cycle]\label{alg:mg}
	
	Initial guess: $u^L\in V_L$, right-hand side $f_L$
	
	$u_l \leftarrow$ MGVC({\rm int} l, $u_l\in V_l$, $f_l$)
	
\{

\quad 	{\rm if }$(l==0)$ $u_0:=A_0^{-1}f_0$

\quad {\rm else}

\qquad \{

\qquad \quad $u_l\leftarrow S^{\rm hybrid}(A_l, u_l, f_l)$ \qquad [pre-smoothing]

\qquad \quad $\eta_{l-1}\leftarrow 0$

\qquad \quad $\eta_{l-1}\leftarrow$MGVC($l-1$,$\eta_{l-1}, R_l^{l-1}(F_l-A_lu_l)$)

\qquad \quad  $u_l\leftarrow u_l+P_{l-1}^l \eta_{l-1}$

\qquad \quad $u_l\leftarrow S^{\rm hybrid}(A_l,u_l,F_l)$ \qquad [post-smoothing]

\qquad \}


\}
\end{alg}

Here, we slightly abuse notation by applying the above operators and functions in $V_l$ instead of using grid operators and grid functions for the hybrid smoother, since they have a one-to-one correspondence (of course, the grid operators and grid functions mentioned here need to satisfy boundary conditions).

\section{Two-dimensional local Fourier analysis}\label{sec:lfa}
In this section, we intend to utilize the local Fourier analysis (LFA) as a tool to investigate the solvers introduced in Section \ref{sec:solver}. For simplicity, we will focus on the SAFE scheme applied to the two-dimensional $\bm{H}({\rm curl})$ convection-diffusion problem. Given the constraints of LFA, we need to assume that the diffusion coefficient $\varepsilon$, convection field $\bm{\beta}$, and reaction coefficient $\gamma$ are all constants. 

The LFA for a node-based method typically relies on the following modes:
\begin{equation} \label{eq:varphi}
	\varphi(\bm{\theta}):=e^{i\bm{\theta}\cdot\bm{x}/h}:=e^{i{\theta_1x_1+\theta2x_2\over h}}, \quad \bm{\theta}\in\mathbb{R}^2.
\end{equation}
Since $\bm{\theta}$ and $\bm{x}$ are intertwined, for simplicity, we omit explicit representation of the dependence on $\bm{x}$ in the expression. As indicated in \eqref{eq:stencil-SAFE}, there exist two distinct stencils for horizontal and vertical edges. Hence, following the notation in \cite{boonen2008local}, we decompose the Fourier mode along the direction, forming a two-dimensional space: 
\begin{equation}\label{eq:2dmode}
F_E(\bm{\theta}) := \mathrm{span}\{\varphi_1(\bm{\theta})_E, \varphi_2(\bm{\theta})_E\},  \quad \varphi_i(\bm{\theta})_E:=\chi_{E_i} \varphi(\bm{\theta})_E,\quad i=1,2,
\end{equation}
where $\chi_{E_i}$ stands for the indicator function. 

Then, the edge-based Fourier mode $\varphi(\bm{\theta})_E$ can be expressed as:
\begin{equation} \label{eq:standard-basis}
	\varphi(\bm{\theta})_E =\Phi(\bm{\theta})\left[\begin{aligned}
		\alpha_1\\\alpha_2
	\end{aligned}\right],\quad \text{ with } \Phi(\bm{\theta}) := [\varphi_1(\bm{\theta})_E\quad \varphi_2(\bm{\theta})_E].
\end{equation}

It is shown in  \cite[Theorem 3.1]{boonen2008local} that $F_E(\bm{\theta})$ remains invariant for any edge operator $L$. The $2\times 2$ representation matrix under $\Phi(\bm{\theta})$ is denoted as $L_{\Phi}(\bm{\theta})$, serving as a fundamental quantity in LFA.


Note that the frequency has a period of $2\pi$. We confine the frequency to the interval $\bm{\Theta} := [-\pi/2, 3\pi/2)^2$. Since smoothers typically aim to diminish high-frequency error components, their effectiveness can be assessed by examining their behavior at high frequencies. We denote the low and high frequency intervals as follows:
$$
\bm{\Theta}^{\rm low}:=[-\pi/2,\pi/2),\quad \bm{\Theta}^{\rm high}:=\bm{\Theta}\backslash \bm{\Theta}^{\rm low}.
$$

\subsection{The splitting in Fourier space} In \eqref{eq:2dmode}, it is observed that $\{\varphi_i(\bm{\theta})_E\}_{i=1,2}$ provides a basis representation for $F_E(\bm{\theta})$. In this section, we will present an alternative basis representation for $F_E(\bm{\theta})$ under which the matrix representation of the edge operator $K$ in \eqref{eq:stencil-SAFE} becomes diagonal. Such a splitting in Fourier space essentially embodies a discrete Helmholtz-type decomposition of $\bm{H}({\rm curl})$ involving the effect of convection.

We denote $J^{\rm grad}_{\varepsilon,\bm{\beta}}: \mathcal{N}\rightarrow\mathcal{E}$ and $C: \mathcal{E}\rightarrow\mathcal{F}$ as the grid operators of the $\mathcal{J}^{\rm grad}_{\varepsilon,\bm{\beta},h}$ and ${\rm curl}$ operators, respectively. They take the following form:
$$
\begin{aligned}
(J^{\rm grad}_{\varepsilon,\bm{\beta}} f)_{\bm{k}} &=
\begin{cases}
-B_{\varepsilon}(b_1)f_{k_1-\frac12,k_2} + B_{\varepsilon}(-b_1) f_{k_1+\frac12,k_2} & \text{if } \bm{k} \in E_1, \\
-B_{\varepsilon}(b_2)f_{k_1,k_2-\frac12} + B_{\varepsilon}(-b_2) f_{k_1,k_2+\frac12} & \text{if } \bm{k} \in E_2, 
\end{cases}\quad \forall f \in \mathcal{N}, \\
(C^T f)_{\bm{k}} &= 
\begin{cases}
-f_{k_1,k_2-\frac12} + f_{k_1,k_2+\frac12} & \text{if } \bm{k} \in E_1, \\
f_{k_1-\frac12,k_2} - f_{k_1+\frac12,k_2+\frac12}& \text{if } \bm{k} \in E_2, 
\end{cases} \quad \forall f \in \mathcal{F},
\end{aligned}
$$
where we recall that $b_i :=\beta_i h$, and $C^T: \mathcal{F} \to \mathcal{E}$ is the transposed grid operator of $C$.

Now for any $\bm{\theta}\in \bm{\Theta}$, we consider $\varphi(\bm{\theta})_N$ and $\varphi(\bm{\theta})_F$ as the interpolation of $\varphi(\bm{\theta})$ in \eqref{eq:varphi} on the node and face index sets, respectively, see \eqref{eq:grid-function}. Then, we define 
\begin{equation} \label{eq:psi}
\psi_J(\bm{\theta})_{E} := 
\begin{cases}
J^{\rm grad}_{\varepsilon,\bm{\beta}} \varphi(\bm{\theta})_N & \text{if } \bm{\theta} \neq (0,0), \\
J^{\rm grad}_{\varepsilon,\bm{\beta}} \tilde{\chi}_N & \text{if } \bm{\theta} = (0,0),
\end{cases} ~
\psi_s(\bm{\theta})_{E} := 
\begin{cases}
C^T \varphi(\bm{\theta})_F & \text{if } \bm{\theta} \neq (0,0), \\
C^T \tilde{\chi}_F & \text{if } \bm{\theta} = (0,0),
\end{cases}
\end{equation}
where the grid functions $\tilde{\chi}_N$ and $\tilde{\chi}_F$ are defined by $\tilde{\chi} :=\varphi(({\pi},\pi))$. 

\begin{lemma}[Fourier-Helmholtz splitting] \label{lm:Helmholtz-splitting}
For all $\bm{\theta}\in \bm{\Theta}$, it holds that 
$$F_E(\bm{\theta}) = \mathrm{span} \{\psi_J(\bm{\theta})_E, \psi_s(\bm{\theta})_E \}.$$
Specifically, if we denote $\Psi(\bm{\theta}):=[\psi_J(\bm{\theta})_E\quad \psi_s(\bm{\theta})_E]$, it has the following transformation relationship with \eqref{eq:standard-basis}:
$$
\Psi(\bm{\theta}) = \Phi(\bm{\theta})H(\bm{\theta}),
$$
where 
\begin{equation} \label{eq:H}
	H(\bm{\theta}):=
	\begin{cases}
	\begin{bmatrix}
		b_1\cos({\theta_1\over 2})+ i \sin ({\theta_1\over 2})[B_{\varepsilon}(b_1)+B_{\varepsilon}(-b_1)]& 2i\sin \left(\frac{\theta_2}{2}\right) \\
		b_2\cos({\theta_2\over 2})+ i \sin ({\theta_2\over 2})[B_{\varepsilon}(b_2)+B_{\varepsilon}(-b_2)] & -2i\sin \left(\frac{\theta_1}{2}\right)
	\end{bmatrix}
	&\text{ if }\bm{\theta}\neq(0,0),\vspace{1mm}\\
	\begin{bmatrix}
		 i [B_{\varepsilon}(b_1)+B_{\varepsilon}(-b_1)]& 2i\\
		i [B_{\varepsilon}(b_2)+B_{\varepsilon}(-b_2)] & -2i\
	\end{bmatrix}
	&\text{ if }\bm{\theta}=(0,0).
\end{cases}
\end{equation}
\end{lemma}
\begin{proof}
When $\bm{\theta} \neq (0,0)$, using the identity $B_{\varepsilon}(-s) - B_{\varepsilon}(s) = s$ by the definition of Bernoulli function \eqref{eq:bernoulli}, a direct computation from \eqref{eq:psi} yields
$$
	\begin{aligned}
		\psi_J(\bm{\theta})_{\bm{k}} &=\left\{
		\begin{array}{ll}
		-B_\varepsilon(b_1)\varphi(\bm{\theta})_{k_1-\frac{1}{2}, k_2}
		+B_\varepsilon(-b_1)\varphi(\bm{\theta})_{k_1+\frac{1}{2}, k_2} &\text{if } \bm{k} \in E_1,\\
		-B_\varepsilon(b_2)\varphi(\bm{\theta})_{k_1, k_2-\frac{1}{2}}
		+B_\varepsilon(-b_2)\varphi(\bm{\theta})_{k_1, k_2+\frac{1}{2}} 
		  &\text{if }  \bm{k} \in E_2,
		\end{array}
		\right.\\
		&=\left\{\begin{array}{ll}
		(b_1\cos({\theta_1\over 2})+ i \sin ({\theta_1\over 2})[B_\varepsilon(b_1)+B_\varepsilon(-b_1)])\varphi(\bm{\theta})_{\bm{k}} &\text{if } \bm{k} \in E_1,\\
		(b_2\cos({\theta_2\over 2})+ i \sin ({\theta_2\over 2})[B_\varepsilon(b_2)+B_\varepsilon(-b_2)])\varphi(\bm{\theta})_{\bm{k}} &\text{if } \bm{k} \in E_2,\\
		\end{array}\right.
			\end{aligned}
$$
and 
$$
\psi_s(\bm{\theta})_{\bm{k}} =
\begin{cases}
	2 i \sin (\frac{\theta_2}{2}) \varphi(\bm{\theta})_{\bm{k}} & \text { if } \bm{k} \in E_1, \\
	-2 i \sin (\frac{\theta_1}{2}) \varphi(\bm{\theta})_{\bm{k}} & \text { if } \bm{k} \in E_2,
\end{cases}
$$
which leads to \eqref{eq:H}, by easily checking the case in which $\bm{\theta}=(0,0)$. Noting that the Bernoulli function in \eqref{eq:bernoulli} is always positive, it follows that the real part of $\det H(\bm{\theta})$ is always positive. Consequently,  the matrix $H(\bm{\theta})$ is invertible.
\end{proof}

\begin{remark}[discrete Helmholtz-type decomposition]
When the convective field is absent, i.e., $b_1 = b_2 = 0$, we have $B_1(\pm b_i) = 1$. In this case, the above result aligns with the Fourier-Helmholtz splitting presented in \cite[Section 4]{boonen2008local}. In fact, the splitting in Fourier space in \cite{boonen2008local} corresponds to a discrete Helmholtz decomposition in the grid function space $\mathcal{E}$: $\mathcal{E} = \text{ran}(G) \oplus \text{ran}(C^T)$, which is known to be $l^2$-orthogonal. When considering the convective effects, Lemma \ref{lm:Helmholtz-splitting} corresponds to another decomposition $\mathcal{E} = \text{ran}(J^{\text{grad}}_{\varepsilon, \bm{\beta}}) \oplus \text{ran}(C^T)$. As can be easily seen from \eqref{eq:discreteflux-grad}, this decomposition is $l^2$-orthogonal under a discrete exponential weight function.

\end{remark}

Thanks to Lemma \ref{lm:Helmholtz-splitting}, a grid function in $F_E(\bm{\theta})$ represented in terms of the basis $\Psi(\bm{\theta})$ can be converted into a representation with respect to the basis $\Phi(\bm{\theta})$ using the following transformation:
\begin{equation}\label{eq:basistransform}
	\Psi(\bm{\theta})\left[\begin{array}{l}\alpha_J \\\alpha_s
	\end{array}\right]=\Phi(\bm{\theta})\left[\begin{array}{l}\alpha_1\\\alpha_2
	\end{array}\right]\quad\text{ and }\quad 
	H(\bm{\theta})\left[\begin{array}{l}\alpha_J \\\alpha_s
	\end{array}\right]
	=\left[\begin{array}{l}\alpha_1\\\alpha_2
	\end{array}\right].
\end{equation}
The following abbreviations will be used:
$$
\begin{aligned}
s_i = \sin({\theta_i\over 2}), \quad \tilde{s}_i = \sin(\theta_i), \quad c_i = \cos({\theta_i\over 2}), \quad \tilde{c}_i = \cos(\theta_i).
\end{aligned}
$$

In light of the stencil form $K$ in \eqref{eq:stencil-SAFE}, its Fourier representation under the basis $\Phi(\bm{\theta})$ \eqref{eq:standard-basis} has the form:
\begin{equation}\label{eq:Kphi}
 K_{\Phi}(\bm{\theta}) =
{\footnotesize
{1\over h^2}
\begin{bmatrix}
 2s_2^2(B_{\varepsilon}(b_2)+ B_{\varepsilon}(-b_2))-i\tilde{s}_2 b_2 & -2 s_1 s_2(B_{\varepsilon}(b_1)+B_{\varepsilon}(-b_1))+2is_2c_1b_1 \\
		-2 s_1 s_2 (B_{\varepsilon}(b_2)+B_{\varepsilon}(-b_2))+2is_1c_2b_2 & 2 s_1^2(B_{\varepsilon}(b_1)+B_{\varepsilon}(-b_1))-i\tilde{s}_1b_1
\end{bmatrix}.
}
\end{equation}
By applying \eqref{eq:basistransform},  we have
\begin{equation}\label{eq:Kpsi}
	\begin{aligned}
	 K_{\Psi}(\bm{\theta})  &= H(\bm{\theta})^{-1} K_{\Phi}(\bm{\theta}) H(\bm{\theta})\\
		& =
		{\footnotesize
		{1\over h^2}\begin{bmatrix}
			0 & 0 \\
			0 & 2(s_1^2(B_{\varepsilon}(b_1)+B_{\varepsilon}(-b_1))+s_2^2(B_{\varepsilon}(b_2)+B_{\varepsilon}(-b_2))-i(\tilde{s}_1b_1+\tilde{s}_2b_2))
		\end{bmatrix}.
		}
	\end{aligned}
\end{equation}
The structure of $K_\Psi(\bm{\theta})$ allows $\psi_J(\bm{\theta})_E$ and $\psi_s(\bm{\theta})_E$ to be identified as the eigenvectors of the edge operator $K$. It can be seen that although the Fourier representation matrix $K_\Phi(\bm{\theta})$ is non-symmetric, it is diagonalizable based on the Fourier-Helmholtz splitting in Lemma \ref{lm:Helmholtz-splitting}. This enables us to subsequently analyze the specific effects of the smoother on different components.

For the mass component $ M$, it is straightforward to obtain:
\begin{equation} \label{eq:mass-phi}
	 M_\Phi(\bm{\theta})=\frac{1}{3}
	 \begin{bmatrix}
		2+\tilde{c}_2&0\\0&2+\tilde{c}_1
	\end{bmatrix}.
	\end{equation}
Collecting the above results, we can obtain the Fourier representation of $A$ as $A_\Phi(\bm{\theta}) = K_\Phi(\bm{\theta}) +\gamma M_\Phi(\bm{\theta})$.

\subsection{LFA for smoothers}
In this part, we give the LFA for the downwind (lexicographic) Gauss-Seidel smoother and hybrid smoother.

\subsubsection{Downwind Gauss-Seidel smoother}
Taking the splitting $(A^{1,+},A^{1,-})$ as an example, which is effective for convection in the first quadrant. 
The corresponding iteration matrix is given by $S^1:=-(A^{1,+})^{-1}A^{1,-}$. Consequently, the error reduction in $F_E(\bm{\theta})$ can be expressed as:
\begin{equation}
	\left[\begin{array}{c}\alpha_1\\\alpha_2
	\end{array}\right]^{[\nu+1]}\begin{aligned}=
	-(A^{1,+}_\Phi(\bm{\theta}))^{-1}A^{1,-}_\Phi(\bm{\theta})\left[\begin{array}{c}\alpha_1\\\alpha_2
	\end{array}\right]^{[\nu]}
	:=S^1_\Phi(\bm{\theta}) \left[\begin{array}{c}\alpha_1\\\alpha_2
	\end{array}\right]^{[\nu]}.
\end{aligned}
\end{equation}

Similarly, we can obtain the representation of $A^{i,+}_\Phi(\bm{\theta})$, where $i=2,3,4$, to derive $S^{i}_\Phi(\bm{\theta})$. 
Noting that the complete downwind smoother in Algorithm \ref{alg:downwind} combines these four smoothers together, allowing the error equation to be analyzed using the spectral radius of the downwind smoother $S^{\rm dw}_\Phi(\bm{\theta}):=S^4_\Phi(\bm{\theta})S^3_\Phi(\bm{\theta})S^2_\Phi(\bm{\theta})S^1_\Phi(\bm{\theta})$, i.e.,
$$
\rho^{\rm dw}(\bm{\theta}):=\rho(S^{\rm dw}_\Phi(\bm{\theta}))=\max\{|\lambda|:\lambda\in \sigma(S^{\rm dw}_\Phi(\bm{\theta}))\}.
$$
Using the relationship \eqref{eq:basistransform}, we can also express this representation under the basis $\Psi(\bm{\theta})$ as $S^{\rm dw}_\Psi(\bm{\theta}):=H^{-1}(\bm{\theta}) S^{\rm dw}_\Phi(\bm{\theta})H(\bm{\theta})$. 

To investigate the properties of downwind smoother, we set $\bm{\beta}=(\frac{1}{2},{\sqrt{3}\over 2})^T$, $\gamma=1$ and $h=1/32$, and present the  contour of $\rho^{\rm dw}(\bm{\theta})$ for different $\varepsilon$. 
The bold line separates the low frequency region $\bm{\Theta}^{\rm low}$ in the lower left corner from the high frequency region $\bm{\Theta}^{\rm high}$.  

As shown in Figure \ref{fig:smoother-h32}, in diffusion-dominated cases ($\varepsilon=1$), it is obvious that regardless of the high-frequency or low-frequency regions, the corresponding smoothing factor is very close to $1$, indicating the ineffectiveness of a single downwind smoother. This finding is similar to the result for Maxwell's equations \cite{hiptmair1998multigrid,boonen2008local}. As diffusion diminishes and convection gradually dominates, the spectral radius in the high-frequency region decreases. Meanwhile, it can be observed that its performance in the low-frequency region is relatively poor in the convection-dominated case $(\varepsilon=10^{-4})$.

To delve deeper into the behavior of the smoother, we examine the values of the upper left $S_{\Psi}^{\rm dw}(\bm{\theta})_{1,1}$ and lower left $S_{\Psi}^{\rm dw}(\bm{\theta})_{2,1}$ corners of the representation matrix $S_{\Psi}^{\rm dw}(\bm{\theta})$ under the basis $\Psi(\bm{\theta})$, as shown in Figure \ref{fig:smoother-left}.  It can be observed that for both $\varepsilon=1$ and $\varepsilon=10^{-4}$, the values in the lower left corner are close to zero and much smaller compared to those in the upper left corner. This suggests that the basis function $\psi_J(\bm{\theta})$ is a near-eigenfunction of $S_{\Psi}^{\rm dw}(\bm{\theta})$. Moreover, upon comparison with the results in Figures \ref{fig:smoother-h32}(c) and \ref{fig:smoother-left}(c), the values in the upper left corner are close to the spectral radius of $S_{\Psi}^{\rm dw}(\bm{\theta})$ when $\varepsilon = 10^{-4}$. Therefore, it can be inferred that the convergence behavior of the smoother primarily relies on the error elimination on the kernel $\psi_J(\bm{\theta})$ for the convection-dominated case. 
Therefore, whether diffusion or advection dominates, there is an urgent demand for a hybrid smoother capable of efficiently tackling this phenomenon by effectively mitigating errors in the kernel component, as will be demonstrated later.

\begin{figure}[!htbp]
	\centering
	\subfloat[$\varepsilon=1$]{
		\includegraphics[width=.32\textwidth]{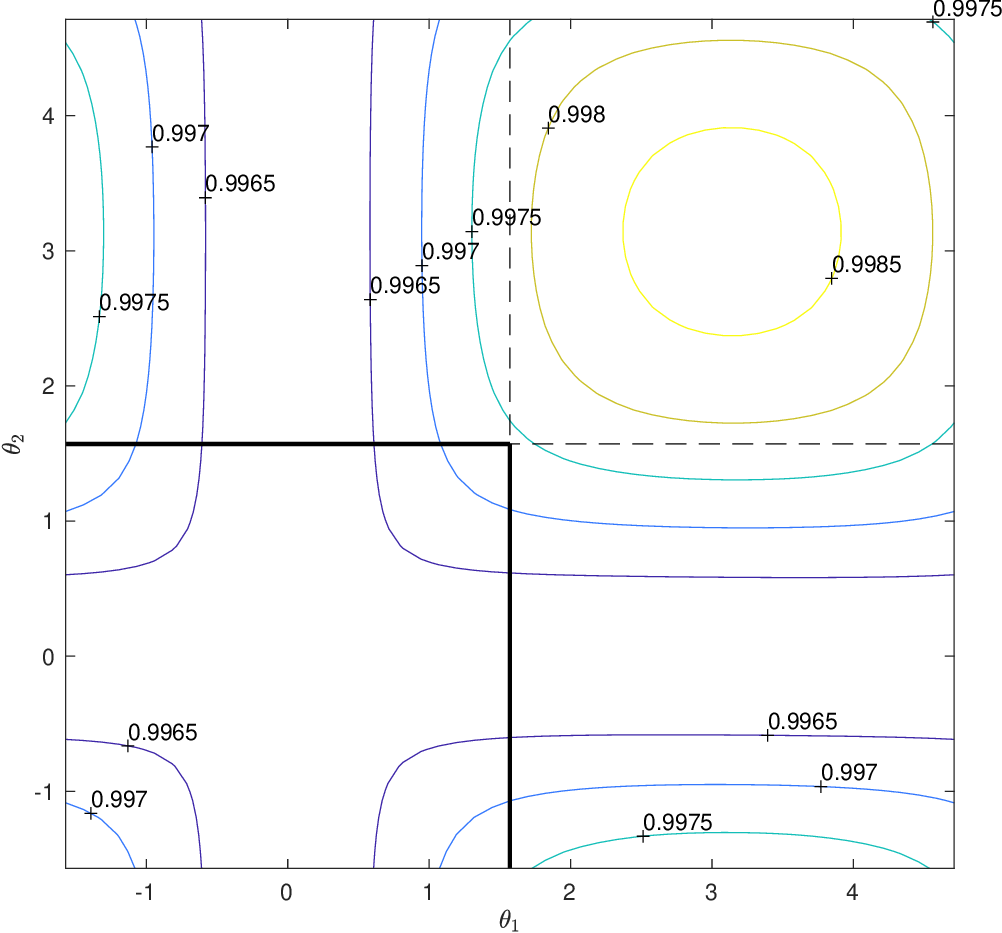}
	}
	\subfloat[$\varepsilon=10^{-2}$]{
		\includegraphics[width=.31\textwidth]{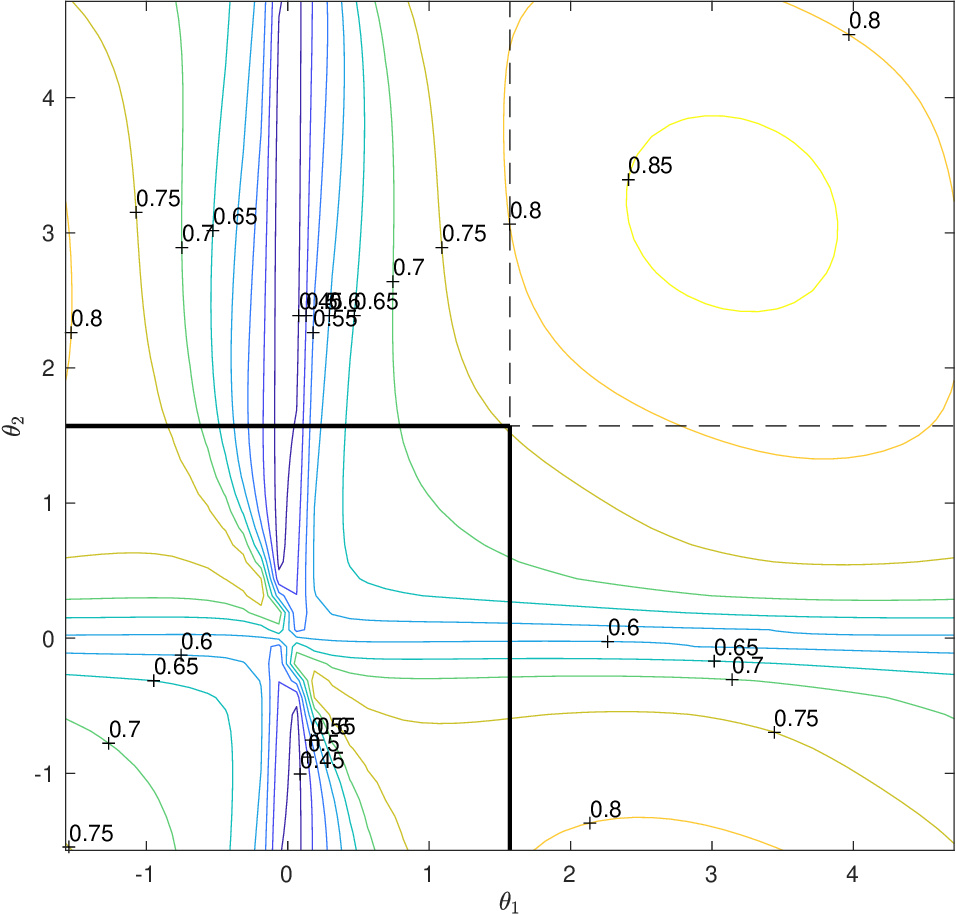}
	}
	\subfloat[$\varepsilon=10^{-4}$]{
	\includegraphics[width=.31\textwidth]{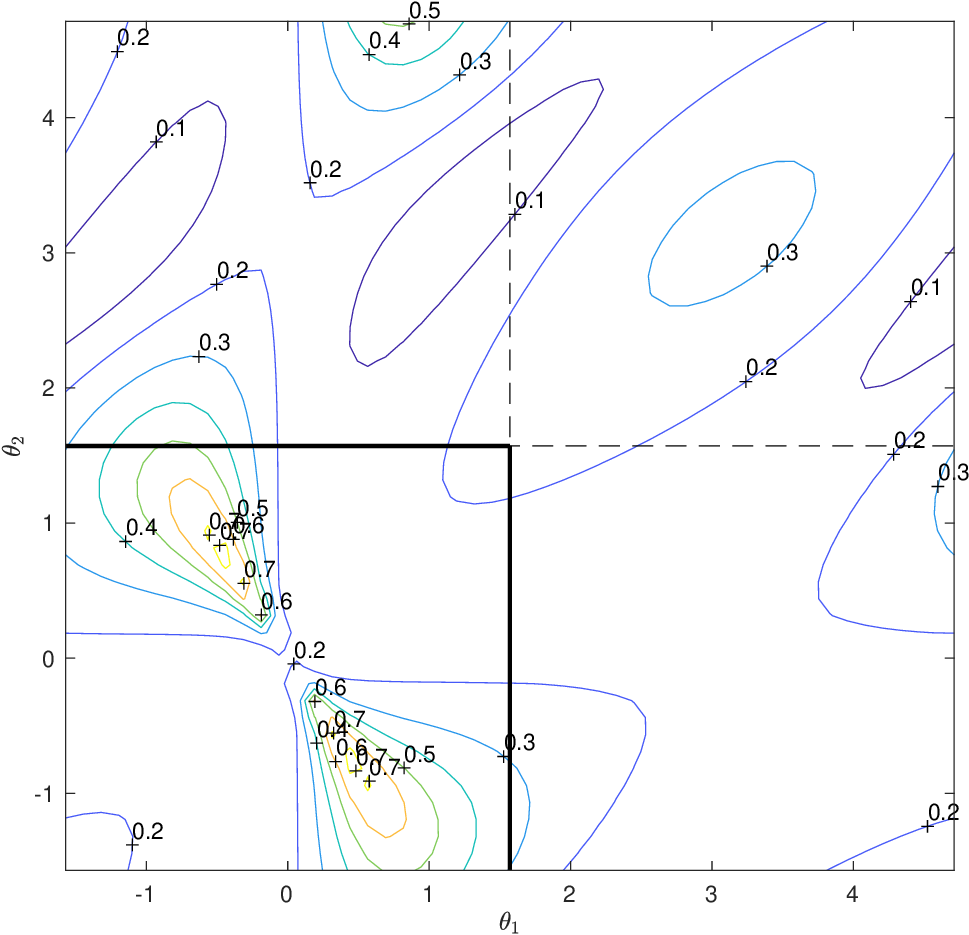}}  
	\\
	\subfloat[$\varepsilon=1$]{
		\includegraphics[width=.31\textwidth]{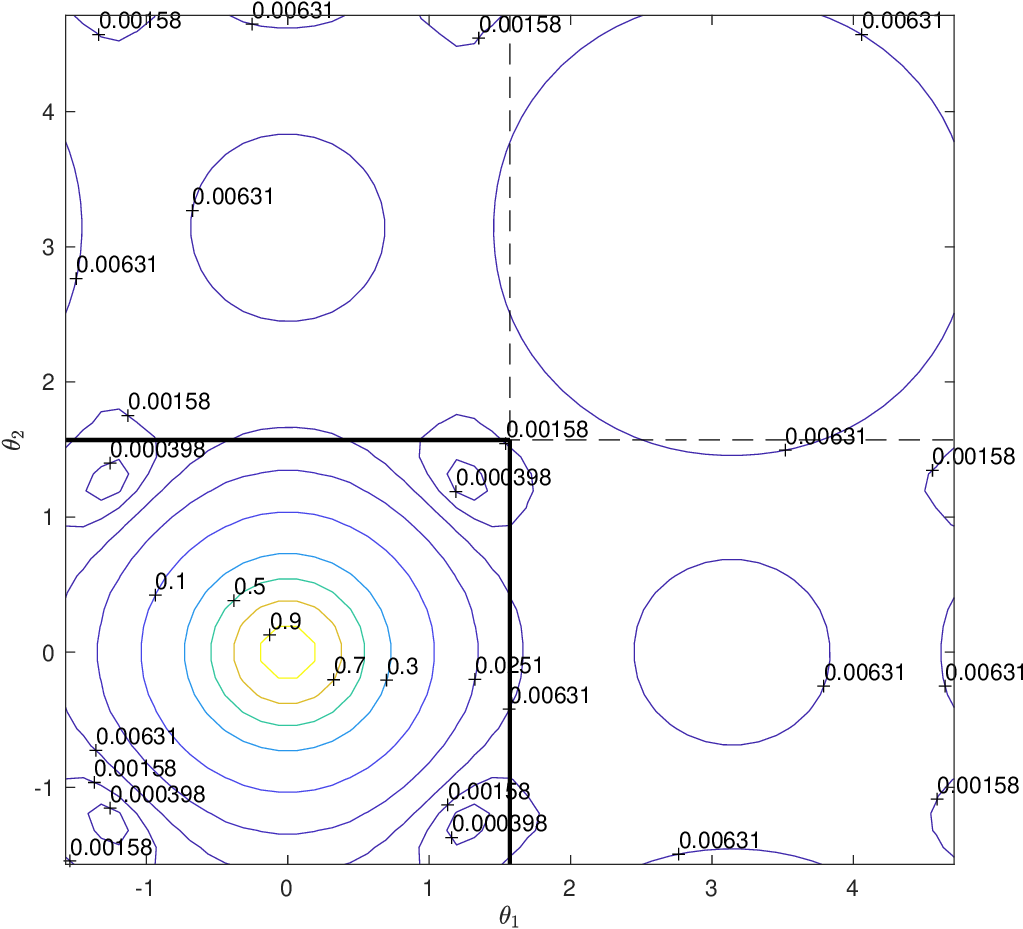}
	}
	\subfloat[$\varepsilon=10^{-2}$]{
		\includegraphics[width=.31\textwidth]{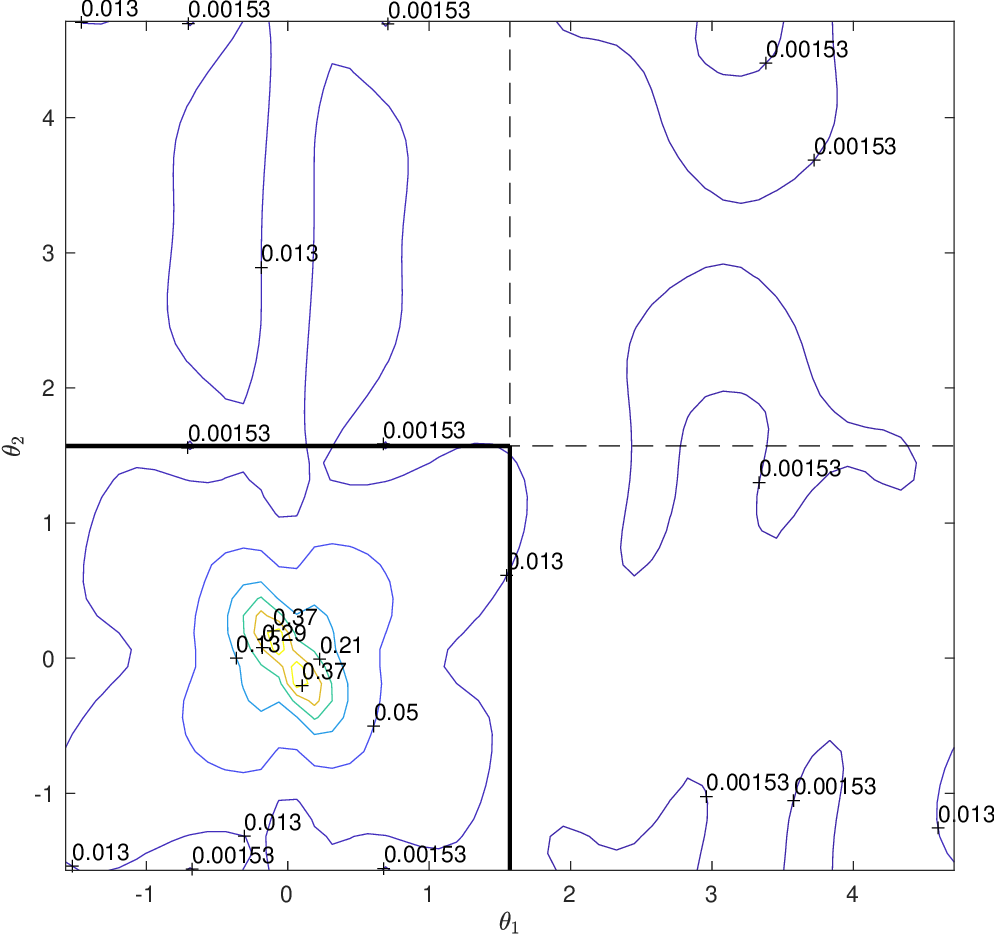}
	}
	\subfloat[$\varepsilon=10^{-4}$]{
		\includegraphics[width=.31\textwidth]{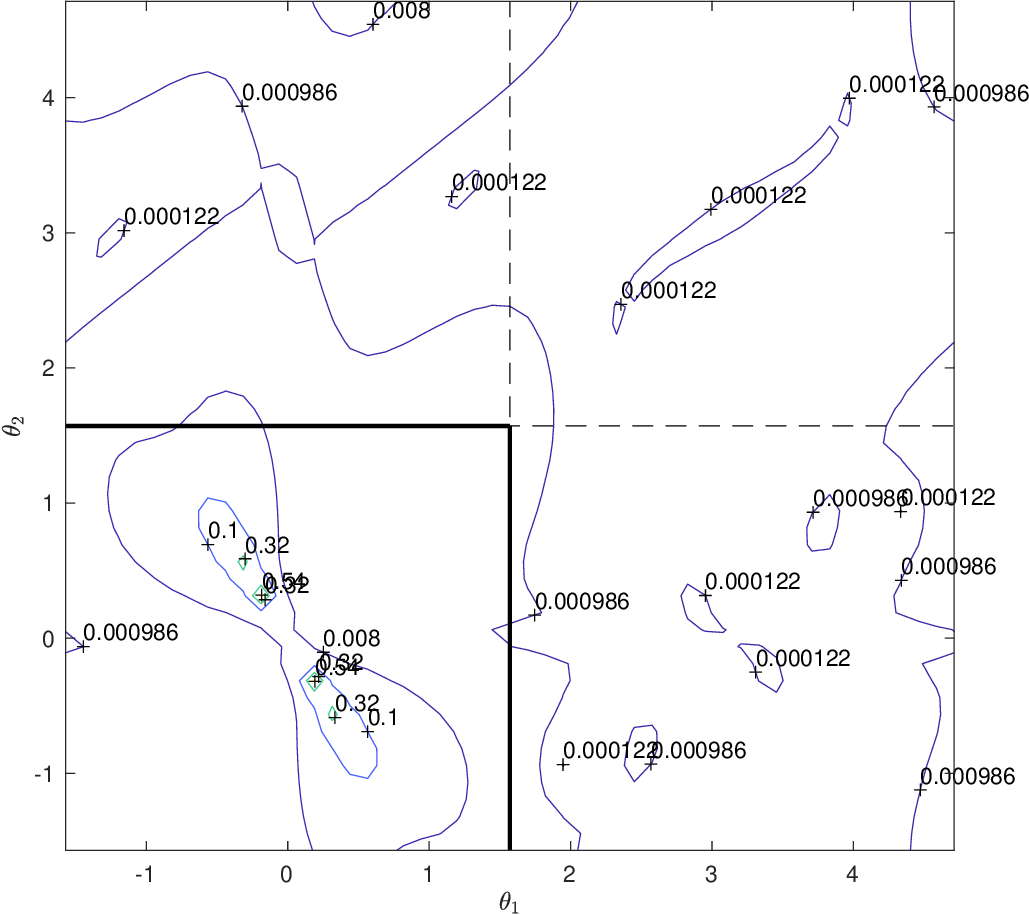}
	}
	\caption{Contour plots of $\rho^{\rm dw}$ (top) and $\rho^{\rm hybrid}$ (bottom) as a function of $\bm{\theta}$ with $h=1/32$, $\bm{\beta}=({1\over 2},{\sqrt{3}\over 2})^T$ and various $\varepsilon=1,10^{-2},10^{-4}$.}
	\label{fig:smoother-h32}
\end{figure}

\begin{figure}[!htbp]
	\centering
		\subfloat[$\varepsilon=1, S_{\Psi}^{\rm dw}(\bm{\theta})_{1,1}$ ]{
			\includegraphics[width=.235\textwidth]{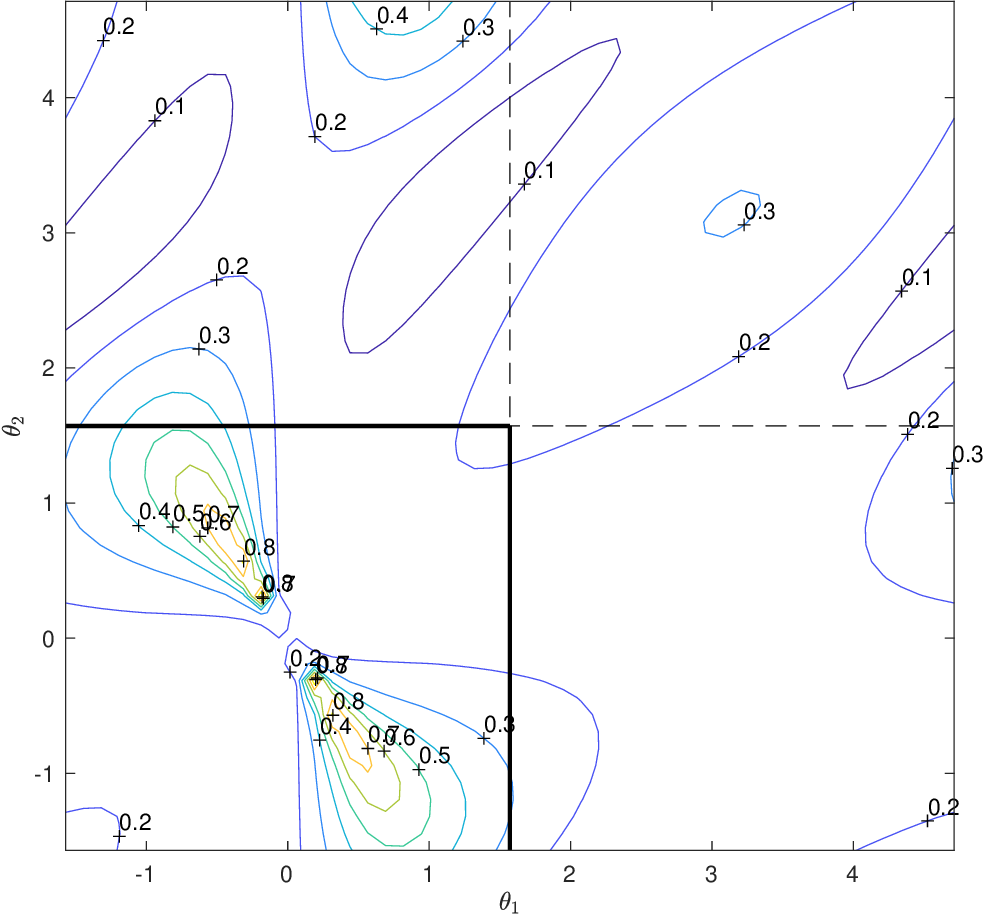}
		}
		\subfloat[$\varepsilon=1, S_{\Psi}^{\rm dw}(\bm{\theta})_{2,1}$ ]{
			\includegraphics[width=.23\textwidth]{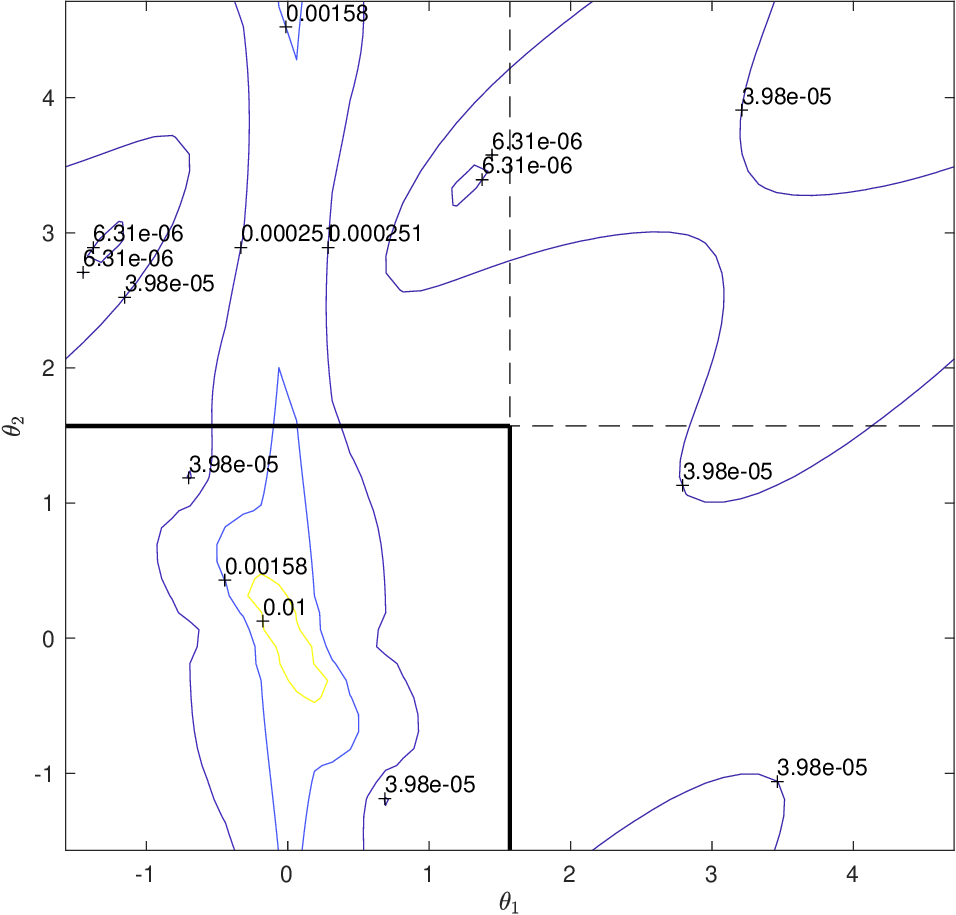}
		}
	\subfloat[$\varepsilon=10^{-4}, S_{\Psi}^{\rm dw}(\bm{\theta})_{1,1}$ ]{
		\includegraphics[width=.235\textwidth]{fig/smoother11_eps_0.0001_beta_30_h_32.eps}
	} 		
	\subfloat[$\varepsilon=10^{-4}, S_{\Psi}^{\rm dw}(\bm{\theta})_{2,1}$ ]{
		\includegraphics[width=.23\textwidth]{fig/smoother21_eps_0.0001_beta_30_h_32.eps}
	}
	\caption{Contour plots of $S_{\Psi}^{\rm dw}(\bm{\theta})_{1,1}$ and $S_{\Psi}^{\rm dw}(\bm{\theta})_{2,1}$ as a function of $\bm{\theta}$ with $h=1/32$, $\bm{\beta}=({1\over 2},{\sqrt{3}\over 2})^T$ and  various $\varepsilon=1$ (left) and $\varepsilon=10^{-4}$ (right).}
	\label{fig:smoother-left}
\end{figure}

\subsubsection{Hybrid smoother}
Firstly, according to the results in \cite[Equ. (6.4)]{boonen2008local}, the Fourier representation of $G^T: \mathcal{E} \to \mathcal{N}$ under $\Phi(\bm{\theta})$ is given by:
$$
   G_\Phi^T(\bm{\theta}) = -2i[s_1\quad s_2],
$$
where we note that the Fourier mode on the node is a classical one-dimensional space spanned by $\varphi(\bm{\theta})_N$. Using \eqref{eq:basistransform}, we can get its representation under basis $\Psi(\bm{\theta})$:
\begin{equation}\label{eq:Gpsi}
\begin{aligned}
 G_\Psi^T(\bm{\theta})& =G^T_\Phi(\bm{\theta})H(\bm{\theta}) \\
 &=
 {\footnotesize
 \begin{bmatrix}
 2s_1^2(B_\varepsilon(b_1)+B_\varepsilon(-b_1))+2s_2^2(B_\varepsilon(b_2)+B_\varepsilon(-b_2))-i(\tilde{s}_1b_1+\tilde{s}_2b_2) & 0
 \end{bmatrix}.
 }
 \end{aligned}
\end{equation}
Thus, by applying \eqref{eq:Kpsi} and \eqref{eq:Gpsi}, we immediately obtain $G_{\Psi}^T(\bm{\theta}) K_{\Psi}(\bm{\theta}) = 0$. Further, the definition of $\psi_J(\bm{\theta})_E$ in \eqref{eq:psi} indicates that the Fourier representation of $J^{\rm grad}_{\varepsilon, \bm{\beta}}$ satisfies $J^{\rm grad}_{\varepsilon, \bm{\beta}, \Psi}(\bm{\theta}) = [1,0]^T$. Now for the auxiliary problem in Algorithm \ref{alg:hybrid}, where \( A^{\rm aux} := G^T A J_{\varepsilon, \bm{\beta}}^{\rm grad} \), we can derive the following result in its Fourier representation.

\begin{theorem}[Fourier representation of auxiliary problem] \label{thm:auxcharacter}
When $\varepsilon$ and $\bm{\beta}$ are constants, the Fourier representation of the $A^{\rm aux}$ is equivalent to that of the node operator given by the SAFE scheme  \eqref{eq:SAFE-grad} for $H({\rm grad})$ convection-diffusion problem scaled by $\gamma$.
\end{theorem}
\begin{proof}
Noting that $A = K + \gamma M$ and $G_{\Psi}^T(\bm{\theta}) K_{\Psi}(\bm{\theta}) = 0$, we have the result following under the basis $\Psi(\bm{\theta})$:
$$
	\begin{aligned}
	&\quad  \gamma^{-1} A^{\rm aux}_{ \Psi}(\bm{\theta}) \\ 
	&= \gamma^{-1} G^T_\Psi(\bm{\theta})A_\Psi(\bm{\theta})J^{\rm grad}_{\varepsilon, \bm{\beta},\Psi}(\bm{\theta})\\
		&= \gamma^{-1} G^T_\Psi(\bm{\theta})(K_\Psi(\bm{\theta})+\gamma M_\Psi(\bm{\theta}))
		J^{\rm grad}_{\varepsilon, \bm{\beta},\Psi}(\bm{\theta})  \\
		& =  G^T_\Psi(\bm{\theta}) M_\Psi(\bm{\theta})J^{\rm grad}_\Psi(\bm{\theta})  \\
		& = 
		{\scriptsize \begin{matrix}
		{2 \over 3}\big[(2+\tilde{c}_2)s_1^2(B_\varepsilon(b_1)+B_\varepsilon(-b_1))+(2+\tilde{c}_1)s_2^2(B_\varepsilon(b_2)+B_\varepsilon(-b_2)) -\frac{i}{2}\tilde{s}_1b_1(2+\tilde{c}_2)-\frac{i}{2}\tilde{s}_2b_2(2+\tilde{c}_1)\big].
		\end{matrix}
		}
	\end{aligned}
$$
Here, we use the Fourier representation matrix of $M$ in \eqref{eq:mass-phi}.
Through a direct computation, we can confirm that above result precisely mirrors the Fourier representation of the node-based stencil \eqref{eq:gradcdstencil}. This completes the proof. 
\end{proof}

Now suppose the smoother applied to the auxiliary problem in Algorithm \ref{alg:hybrid} is a linear operator $S^{\rm aux}$ with  Fourier representation $S^{\rm aux}(\bm{\theta})$. Then, without considering the pre-smoothing and post-smoothing steps in the hybrid smoother, regarding the kernel correction procedure in Algorithm \ref{alg:hybrid}, the iteration matrix is
$$
S^{\rm kc} := I-J^{\rm grad}_{\varepsilon,\bm{\beta}}(I-S^{\rm aux})(A^{\rm aux})^{-1}G^TA.
$$
Using the expression of $A_{\Psi}^{\rm aux}(\bm{\theta})$ in the proof of above Theorem, Fourier representation of $S^{\rm kc}$ under the basis $\Psi(\bm{\theta})$ turns out to be 
\begin{equation} \label{eq:S-kc}
	\begin{aligned}
 & ~~S^{\rm kc}_{\Psi}(\bm{\theta}) :=
I-J^{\rm grad}_{\Psi}(\bm{\theta})(1-S^{\rm aux}(\bm{\theta}))(A^{\rm aux}_\Psi(\bm{\theta})^{-1})G_{\Psi}^T(\bm{\theta})A_\Psi(\bm{\theta})\\
=&{\scriptsize
\begin{bmatrix}
		S^{\rm aux}(\bm{\theta}) & {2(S^{\rm aux}(\bm{\theta})-1)s_1s_2(\tilde{c}_2-\tilde{c}_1)\over (2+\tilde{c}_2)s_1^2(B_{\varepsilon}(b_1)+B_{\varepsilon}(-b_1))+(2+\tilde{c}_1)s_2^2(B_{\varepsilon}(b_2)+B_{\varepsilon}(-b_2))
			-\frac{i}{2}\tilde{s}_1b_1(2+\tilde{c}_2)-\frac{i}{2}\tilde{s}_2b_2(2+\tilde{c}_1)}\\
		0 & 1
\end{bmatrix}.
}
\end{aligned}
\end{equation}
This formulation indicates that the nodal part of this smoother leaves solenoidal errors unaffected (lower right entry). If there are no such errors present, the error in the kernel of $J^{\rm curl}_{\varepsilon, \bm{\beta}}$ is reduced similarly to the error of the auxiliary problem (upper left entry). The upper right entry represents the aliasing contribution of solenoidal errors.

Combined with the downwind GS smoother, we obtain the Fourier representation of hybrid smoother with one downwind pre-smooth and one downwind post-smooth:
\begin{equation} \label{eq:S-hybrid}
	S^{\rm hybrid}_\Psi(\bm{\theta}):= S^{\rm dw}_\Psi(\bm{\theta}) S^{\rm kc}_{\Psi}(\bm{\theta})  S^{\rm dw}_\Psi(\bm{\theta}) .
\end{equation}
The quantity analysis for the smoother is relied on the spectral radius of $S^{\rm hybrid}_\Psi (\bm{\theta})$:
$$
\rho^{\rm hybrid}(\bm{\theta}):=\rho(S^{\rm hybrid}_\Psi(\bm{\theta}))=\max\{|\lambda|:\lambda\in \sigma(S^{\rm hybrid}_\Psi(\bm{\theta}))\}.
$$

To illustrate the efficiency of the hybrid smoother, we utilize a node-based downwind lexicographic Gauss-Seidel discussed  in Section \ref{subsec:kernel}. The results of LFA are depicted in the bottom row of Figure \ref{fig:smoother-h32}.

In diffusion-dominated cases, similar to the hybrid smoother for Maxwell's equations \cite{hiptmair1998multigrid, boonen2008local}, the behavior of the smoother in high-frequency regions is characterized by a small spectral radius, resulting in effective smoothing. Conversely, in low-frequency regions, the spectral radius is close to 1, indicating ineffective smoothing for low-frequency errors.

In convection-dominated cases, the spectral radius of the smoother is close to zero, around $10^{-3}$, in high-frequency regions, indicating good performance. The spectral radius in low-frequency regions in convection-dominated cases is also relatively smaller. This is because the downwind Gauss-Seidel smoother we employ becomes an exact solver for the operator $K$ when diffusion completely degenerates. Considering the overall situation, where the spectral radius in the low-frequency region is still relatively large, it is necessary to employ multigrid methods to effectively reduce low-frequency errors.

\subsection{LFA for two-level method}
Regarding the LFA for two-level method ($L=1$ in Algorithm \ref{alg:mg}), it is relatively standard compared to \cite[Section 7-8]{boonen2008local}. The main difference between Algorithm \ref{alg:mg} and \cite{boonen2008local} is the use of the SAFE scheme for coarse-grid discretization instead of the Galerkin product. The reasons for this are primarily twofold: (i) the bilinear form in the SAFE scheme is not nested due to the nonlinear Bernoulli function; (ii) the downwind smoother heavily depends on the form of the bilinear shape in the current grid. Therefore, we will only outline the main process of the LFA, highlighting the differences from \cite{boonen2008local}.

To investigate the elimination of error in low frequencies, 
it is necessary to introduce the aliasing of low frequency $\bm{\theta} = (\theta_1, \theta_2) \in \bm{\Theta}^{\rm low}$:
$$
    \bm{\theta}^{(a,b)} := (\theta_1 + a\pi, \theta_2 + b\pi) \quad a, b \in \{0,1\}.
$$
This is because the corresponding Fourier modes have the same expression on a grid scale of $2h$. Additionally, the space $F_E(\bm{\theta})$ is not invariant under the two-grid operator, but it is invariant in the eight-dimensional space of $2h$-harmonics as follows:
\begin{equation} \label{eq:2h-harmonics}
    F_8(\bm{\theta}) := \bigcup_{a, b \in \{0,1\}} F_E(\bm{\theta}^{(a, b)}) \quad \bm{\theta} \in \bm{\Theta}^{\text{low}}.
\end{equation}

We recall the classical iteration operator  of two grid method:
\begin{equation}\label{eq:2mgiter}
	C := S(I-PA_c^{-1}RA)S,
\end{equation}
where $S$ is the hybrid smoother proposed in Section \ref{subsec:kernel}, $A:= K + \gamma M$ is the discrete operator on the fine grid. Their Fourier representations on $F_8(\bm{\theta})$ are
\begin{equation} \label{eq:Fourier-SA}
\begin{aligned}
\mathcal{S}(\bm{\theta}) &=
{\scriptsize
\begin{bmatrix}
	S^{\rm hybrid}_\Phi(\bm{\theta}^{(0,0)})&&&\\&S^{\rm hybrid}_\Phi(\bm{\theta}^{(1,0)})&&\\
	&&S^{\rm hybrid}_\Phi(\bm{\theta}^{(0,1)})&\\&&&S^{\rm hybrid}_\Phi(\bm{\theta}^{(1,1)})
\end{bmatrix} 
}
\in \mathbb{C}^{8\times 8}, \\
\mathcal{A}(\bm{\theta}) &= 
{\scriptsize
\begin{bmatrix}
	A_\Phi(\bm{\theta}^{(0,0)})&&&\\&A_\Phi(\bm{\theta}^{(1,0)})&&\\
	&&A_\Phi(\bm{\theta}^{(0,1)})&\\&&&A_\Phi(\bm{\theta}^{(1,1)})
\end{bmatrix} 
}
\in \mathbb{C}^{8\times 8}.
\end{aligned}
\end{equation}
Here, $S^{\rm hybrid}_\Phi(\bm{\theta})$ can be obtained through the transformation on $S^{\rm hybrid}_\Psi(\bm{\theta})$ in \eqref{eq:S-hybrid}, i.e., $S^{\rm hybrid}_\Phi(\bm{\theta}) = H(\bm{\theta})S^{\rm hybrid}_\Psi(\bm{\theta})H^{-1}(\bm{\theta})$, where $H(\bm{\theta})$ is given in \eqref{eq:H}. We also recall the Fourier representation $A_{\Phi}(\bm{\theta}) = K_{\Phi}(\bm{\theta}) + \gamma M_{\Phi}(\bm{\theta})$, see \eqref{eq:Kphi} and \eqref{eq:mass-phi}.

In \eqref{eq:2mgiter}, the restriction operator $R$ and the prolongation operator $P$ are canonical. Their Fourier representations in $F_8(\bm{\theta})$ are provided in \cite[Section 7.2-7.3]{boonen2008local}. Here, we list their expressions:
\begin{equation} \label{eq:Fourier-RP}
\begin{aligned}
\mathcal{R}(\bm{\theta}) &= 
\begin{bmatrix}
R_{\Phi}(\bm{\theta}^{(0,0)}) & R_{\Phi}(\bm{\theta}^{(1,0)}) & R_{\Phi}(\bm{\theta}^{(0,1)}) & R_{\Phi}(\bm{\theta}^{(1,1)})
\end{bmatrix} 
\in\mathbb{C}^{2\times 8}, \\
\mathcal{P}(\bm{\theta}) &= 
\begin{bmatrix}
	P_{\Phi}(\bm{\theta}^{(0,0)})^T & P_{\Phi}(\bm{\theta}^{(1,0)})^T & P_{\Phi}(\bm{\theta}^{(0,1)})^T & P_{\Phi}(\bm{\theta}^{(1,1)})^T
\end{bmatrix}^T \in \mathbb{C}^{8\times 2}.
\end{aligned}
\end{equation}
Here, with the shorthand $s_{i,a} := \sin({\theta_i+a\pi\over 2})$, $\tilde{s}_{i,a} := \sin(\theta_i+a\pi)$, $c_{i,a} := \cos({\theta_i+a\pi\over 2})$ and $\tilde{c}_{i,a} := \cos(\theta_i+a\pi)$,
$$
	R_\Phi(\bm{\theta}^{(a,b)}) = 4P_\Phi(\bm{\theta}^{(a,b)})=
	\begin{bmatrix}
		2c_{1,a}^2c_{2,b}^2(-1)^a & 0\\0&2c_{2,b}^2c_{1,a}^2(-1)^b
	\end{bmatrix} \quad a,b \in \{0,1\}.
$$

The operator $A_c$ on the coarse grid is directly obtained using the SAFE scheme on $\mathcal{T}_{2h}$.  Its corresponding Fourier representation is
\begin{equation} \label{eq:Fourier-Ac}
\mathcal{A}_c(\bm{\theta}^{(0,0)}):=\tilde{{A}}_\Phi(2\bm{\theta}^{(0,0)})\in \mathbb{C}^{2\times 2},
\end{equation}
where $\tilde{{A}}_\Phi(\bm{\theta})$ is defined by replacing $h$ in $A_\Phi(\bm{\theta})$ with $2h$
(Note that the argument $b_i = h\beta_i$ of the Bernoulli function depends on the grid scale $h$).

Combining \eqref{eq:Fourier-SA}, \eqref{eq:Fourier-RP}, and \eqref{eq:Fourier-Ac}, we obtain the Fourier representation of the two-grid cycle in \eqref{eq:2mgiter}:
\begin{equation}\label{eq:Fourier-2mg}
	\mathcal{C}(\bm{\theta}):=\mathcal{S}(\bm{\theta})(I-\mathcal{P}(\bm{\theta})\mathcal{A}^{-1}_c(\bm{\theta})\mathcal{R}(\bm{\theta})\mathcal{A}(\bm{\theta})) \mathcal{S}(\bm{\theta}) \in \mathbb{C}^{8\times 8}.
\end{equation}
Then the asymptotic convergence factor can be calculated as
\begin{equation}
\rho^{\rm mg} := \sup \{\rho(\mathcal{C}(\bm{\theta})):\bm{\theta}\in \bm{\Theta}^{\rm low}\}.
\end{equation}

Figure \ref{fig:2grid} shows the asymptotic convergence factors $\log(\rho^{\rm mg})$ of the two-level method as a function of $\varepsilon$, given different values of $h$ and $\bm{\beta}$. As depicted in the figure, for different directions of convection $\bm{\beta}$, the variation in the convergence factor is similar. It can be observed that with a decrease in $h$, the convergence factor increases but remains consistently bounded. We can see that when $\varepsilon > 0.01$, this factor is close to 0. As $\varepsilon$ decreases, the factor increases but flattens when $\varepsilon$ approaches zero. Overall, regardless of the values of $\varepsilon$ and $h$, the convergence factor is bounded above by 0.3, indicating that the multigrid method exhibits good convergence.

\begin{figure}[!htbp]
	\centering
		\subfloat[$\bm{\beta}=(-{1\over 2},{\sqrt{3}\over 2})$]{
		\includegraphics[width=.32\textwidth]{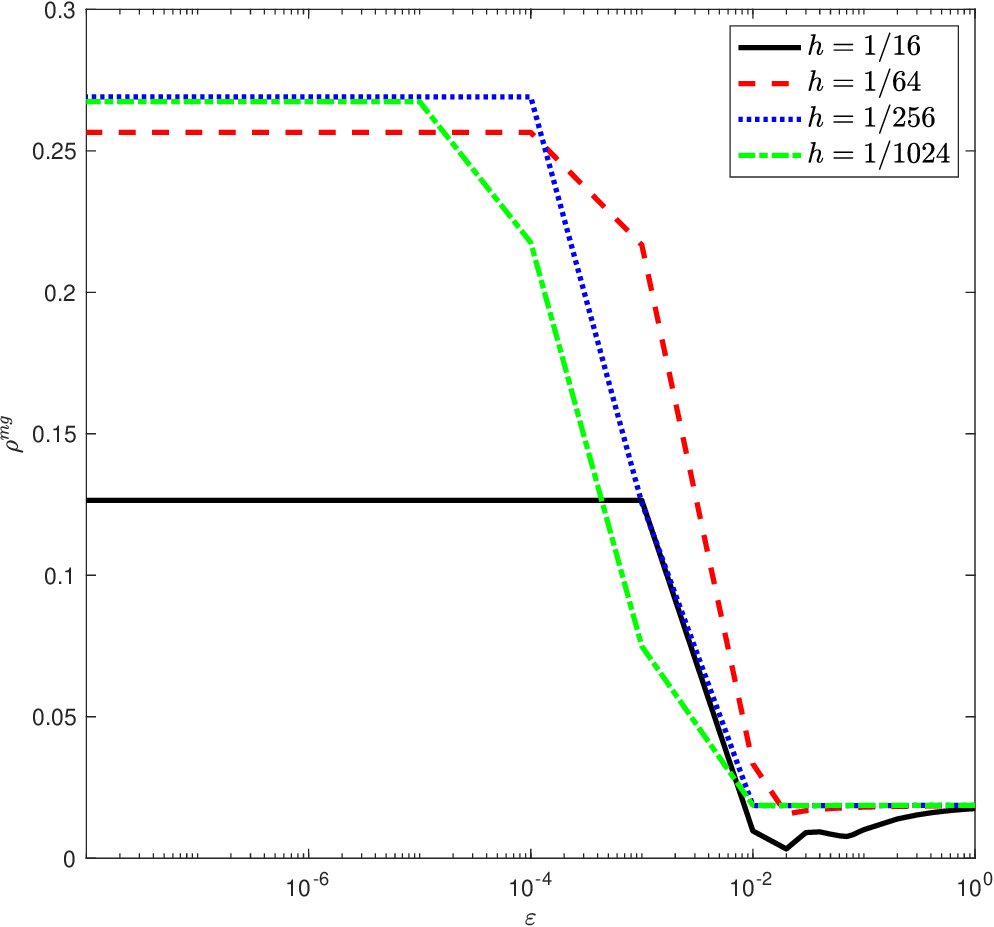}
	}
		\subfloat[$\bm{\beta}=(-{\sqrt{3}\over 2},-{1\over 2})$]{
		\includegraphics[width=.32\textwidth]{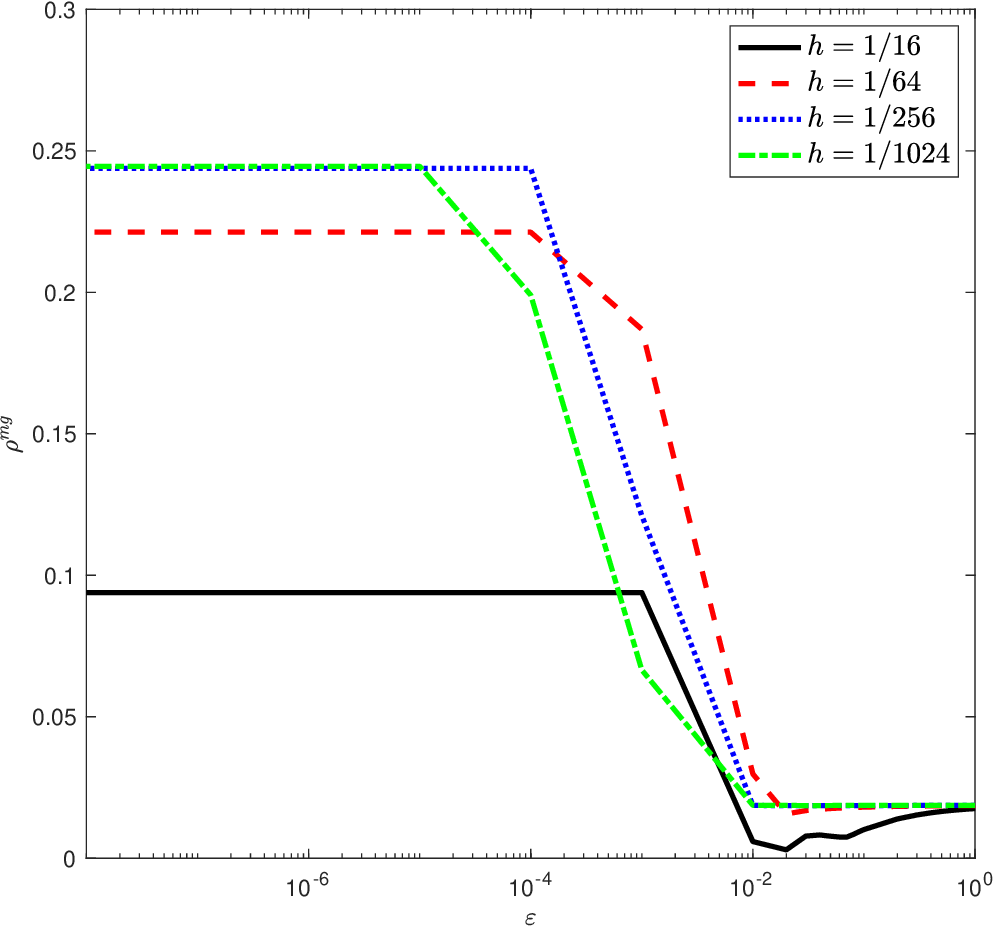}
	}
	\caption{Plots of asymptotic convergence factor $\rho^{\rm mg}$ as a function of $\varepsilon$ given different $h$ and $\bm{\beta}$.}
	\label{fig:2grid}
\end{figure}

\begin{remark}
	For $\bm{\theta} = (0,0)$,  the above derivation does not hold because $\mathcal{A}_c(\bm{\theta})$ is singular. This is a consequence of the absence of boundary conditions, which is generic for a local Fourier analysis.
\end{remark}

\section{Numerical tests}\label{sec:numerical}
In this section, we present some numerical results to illustrate the performance of solvers for $\bm{H}(\rm curl)$ convection-diffusion problems in two dimensions. The meshes used in the experiments are nested and regular, uniformly rectangular partition of $\Omega=(0,1)^2$. We set $\bar{\varepsilon}|_T = \varepsilon(\bm{x}_c|T)$ and $\bar{\bm{\beta}}|_T = \bm{\theta}(\bm{x}_c|_T)$ on each element $T$, and $\bar{\bm{\beta}}|_E = \bm{\theta}(\bm{x}_c|_E)$ on each edge $E$ for the evaluation of operator $J^{\rm grad}_{\bar{\varepsilon},\bar{\bm{\beta}},h}$. The strategy for selecting 
$\bar{\varepsilon}|_E$ in the presence of jump diffusion coefficients will be provided later.

\subsection{Performance of MG}
We present some results demonstrating the performance of Algorithm \ref{alg:mg} for 2D $\bm{H}(\rm curl)$ convection-diffusion problems. We provide the results of both MG iteration and the GMRES iteration using MG as a preconditioner. The termination criterion for the algorithm is set as $\|{r^{[\nu]}}\|/\|{r}^{[0]}\| \le 10^{-8}$, where ${r}^{[\nu]}:=b-Ax^{[\nu]}$ represents the residual at the $\nu$-th iteration. In this subsection, we set $\bm{\beta}=(x_2-0.25,0.75-x_1)^T$ and $\gamma=1$.

\paragraph{Example 1 (constant diffusion coefficients)} 

Table \ref{tab:itermg1} shows the correspond iteration numbers of MG and PGMRES for various $\varepsilon$. We can observe that the iteration number is level independent and the number of iteration in convection dominated case slightly increases. 

\begin{table}[!htbp]
	\centering
	\caption{Number of iteration for example 1 with various $\varepsilon$.}
			\setlength{\tabcolsep}{3.2mm}
	\begin{tabular}{|c|c|ccccc|}
		\hline
		\multicolumn{2}{|c|}{Level}  & 6 & 7 & 8 & 9 & 10\\
		\hline
		\multicolumn{2}{|c|}{DOFs}  & 8,320 & 33,024 &131,584 & 525,312 & 2,099,200\\
		\cline{1-7}
		\multirow{2}*{$\varepsilon=1$}
	& MG & 4& 4& 4& 4 & 4\\
	& PGMRES & 4& 4& 4& 4&4\\
	\hline
	\multirow{2}*{$\varepsilon=10^{-2}$}
	& MG & 4& 4& 4& 4 & 4\\
	& PGMRES & 4& 4& 4& 4&4\\
	\hline
	\multirow{2}*{$\varepsilon=10^{-4}$}
	& MG & 4& 5& 4& 5& 6\\
	& PGMRES & 4& 6& 5& 6&7\\
		\hline
	\end{tabular}
	\label{tab:itermg1}
\end{table}

\begin{table}[!htbp]
	\centering
	\caption{Number of iteration for example 1 using downwind smoother only.}
	\setlength{\tabcolsep}{4mm}
	\begin{tabular}{|c|c|ccccc|}
		\hline
		\multicolumn{2}{|c|}{Level}  & 6 & 7 & 8 & 9 & 10\\
		\hline
		\multirow{2}*{$\varepsilon=1$}
		& MG & $>$100 & $>$100  & $>$100&  $>$100 & $>$100\\
		& PGMRES & 84& $>$100 & $>$100 & $>$100 & $>$100\\
		\hline
		\multirow{2}*{$\varepsilon=10^{-2}$}
		& MG & $>$100& $>$100&$>$100 & $>$100 & $>$100\\
		& PGMRES& 26 & 37&48 &  63 & 85\\
		\hline
		\multirow{2}*{$\varepsilon=10^{-4}$}
		& MG & 38& 71& $>$100& $>$100 & $>$100\\
		& PGMRES & 19& 35& 59 & 98 & $>$100\\
		\hline
	\end{tabular}
	\label{tab:itermg1-1}
\end{table}

\begin{table}[!htbp]
	\centering
	\caption{Number of iteration for example 1  using hybrid smoother in \cite{hiptmair1998multigrid}.}
	\setlength{\tabcolsep}{4mm}
	\begin{tabular}{|c|c|ccccc|}
		\hline
		\multicolumn{2}{|c|}{Level}  & 6 & 7 & 8 & 9 & 10\\
		\hline
		\multirow{2}*{$\varepsilon=1$}
		& MG & 4& 4& 4& 4 & 4\\
		& PGMRES & 4& 4& 4& 4&4\\
		\hline
		\multirow{2}*{$\varepsilon=10^{-4}$}
		& MG & -& -& -& - & -\\
		& PGMRES & 27& $>$100& $>$100 & $>$100&$>$100\\
		\hline
		\multicolumn{2}{l}{- : not converge}
	\end{tabular}
	\label{tab:itermg1-2}
\end{table}

In Table \ref{tab:itermg1-1}, we present the number of iterations for the MG algorithm using only the downwind smoother without the auxiliary problem correction. It can be observed that both iterations converge slowly. While in convection-dominated cases, the number of PGMRES iterations is smaller than that in diffusion-dominated cases, the iterations increase under mesh refinement and is significantly larger than the results obtained using the complete hybrid smoother.  This implies the necessity of a complete hybrid smoother utilizing the auxiliary problem for correction  in $\bm{H}(\rm curl)$ convection-dominated problems.

In Table \ref{tab:itermg1-2}, we display the results when using the hybrid smoother for Maxwell's equations \cite{hiptmair1998multigrid}, where $J^{\rm grad}_{\bar{\varepsilon},\bar{\bm{\beta}},h}$ in Algorithm \ref{alg:hybrid} is replaced with $G$ (matrix representation of gradient operator). As expected, in diffusion-dominated case $\varepsilon=1$, it can have a good  performance. However, in the convection-dominated case ($\varepsilon=10^{-4}$), the algorithm using MG iteration does not converge and the PGMRES iteration converges very slow. This also indicates that for the convection-diffusion problems, it is appropriate to design an asymmetric form given in algorithm \ref{alg:hybrid}.


\paragraph{Example 2 (jump diffusion coefficients)} In this example, we consider problem \eqref{eq:curlcd} with diffusion coefficients that exhibit jumps:
$$
\varepsilon(\bm{x})=
\begin{cases}
	1 & \text{if } x_1\le0.5,\\
	\varepsilon_1 & \text{if }x_1> 0.5.
\end{cases}
$$

\begin{table}[!htbp]
	\centering
	\caption{Number of iteration for example 2 (jump diffusion coefficients).}
	\setlength{\tabcolsep}{6mm}
	\begin{tabular}{|c|c|ccccc|}
		\hline
		\multicolumn{2}{|c|}{Level}  & 6 & 7 & 8 & 9 & 10\\
		\hline
		\multirow{2}*{$\varepsilon_1=0.1$}
& MG & 4& 4& 4& 4 & 4\\
		& PGMRES & 4& 4& 4& 4&4\\
		\hline
		\multirow{2}*{$\varepsilon_1=0.01$}
		& MG & 5& 5& 5& 4 & 4\\
		& PGMRES& 5& 5& 5& 5&5\\
		\hline
		\multirow{2}*{$\varepsilon_1=0.001$}
		& MG & 5& 5& 5& 4 & 4\\
		& PGMRES & 7& 7& 7 & 6 &6\\
		\hline
	\end{tabular}
	\label{tab:itermg2}
\end{table}
Considering the region where the jump exists, we propose choosing the larger part of $\varepsilon$ for the implementation of $J^{\rm grad}_{\bar{\varepsilon},\bar{\bm{\beta}},h}$, which may lead to some improvement. For a detailed discussion of this choice, we refer to \cite{dryja1996multilevel}. Table \ref{tab:itermg2} displays the corresponding iteration numbers for MG and PGMRES. It is observed that with increasing jumps, the number of PGMRES iterations experiences a slight increase. Furthermore, the iteration numbers for both MG and PGMRES remain uniform with respect to the mesh level. Overall, in this example featuring both convection-dominated and diffusion-dominated regions, our proposed MG algorithm is effective and exhibits robust performance under mesh refinement.

\subsection{Convergence factor}
The asymptotic convergence factors calculated by means of the LFA in this paper are checked numerically by applying two-level MG iterations for a zero right-hand side and a random initial guess, which allows the norm of the residual to be reduced much more than with a random right-hand side. 

Table \ref{tab:factor} gives the numerical convergence factor $\rho^{\rm num}=\|b-Ax^{[\nu]} \| / \|b-Ax^{[\nu-1]}\|$, using the Dirichlet boundary condition for a mesh size of $h=1/128$. The analytical  factor $\rho^{\rm mg}$ is calculated by using LFA with a mesh size of $h=1/64$. The setting of different mesh size is motivated by \cite{rodrigo2017validity} which demonstrates comparability under different boundary conditions. The numerical convergence factor is calculated by averaging the convergence factor of 10 iterations. We observe that the convergence factors obtained align closely with the numerical factors calculated under Dirichlet boundary conditions when $\varepsilon$ is relatively large. For smaller $\varepsilon$ the difference is slightly larger but $\rho^{\rm mg}$ can still provide a promising upper bound for $\rho^{\rm num}$. This indicates that LFA also provides valuable insights into the convergence behavior of the algorithm.

\begin{table}[!htbp]
	\centering
	\caption{Comparison of local Fourier analysis results with multigrid convergence for 2D $\bm{H}(\rm curl)$ convection-diffusion problems with $\bm{\beta}=(\frac{\sqrt{3}}{2},\frac{1}{2})^T$. }
	\setlength{\tabcolsep}{3mm}
	\begin{tabular}{c| ccccccc}
		\hline
		$\varepsilon$& 1 & 0.5 & 0.1& $0.01$ & $10^{-4}$&  $10^{-6}$ & $10^{-8}$ \\
		\hline
		$\rho^{\rm mg}$ & 0.0193 & 0.0192 & 0.0186&0.0418 &0.1322 & 0.1322 & 0.1322\\
		$\rho^{\rm num}$& 0.0194& 0.0193 & 0.0188 & 0.0159& 0.0677 & 0.1087 & 0.1119 \\
		\hline
	\end{tabular}
	\label{tab:factor}
\end{table}

\section{Conclusion and future remarks}\label{sec:conclude}
This paper proposes a stable and efficient solver for 2D $\bm{H}(\text{curl})$ convection-diffusion problems. We believe that its design hinges on the following three ingredients: (i) a smoother that utilizes convection information; (ii) effective smoothing of the kernel space of the convection-diffusion operator; and (iii) a multilevel mechanism capable of handling both diffusion-dominated and convection-dominated regimes. For an exponentially fitted numerical discretization, our designed solver seamlessly integrates these three ingredients. Through local Fourier analysis, where we introduce a (modified) Fourier-Helmholtz splitting (Lemma \ref{lm:Helmholtz-splitting}), we have analyzed the effectiveness and convergence behavior of the multigrid algorithm.

We note that for the 3D $\bm{H}(\text{curl})$ convection-diffusion problem on cubic grids, the three aforementioned ingredients can be directly implemented following the approach described in this paper. The only aspect that requires attention is the structure of the 3D stencil. In fact, by following the same logic as in Section \ref{subsec:downwind}, a downwind smoother for sweeping can be designed. However, for unstructured grids, it is important to note that the design of downwind smoothers for edge-based finite elements becomes more complicated due to the intricate structure.

For the Maxwell's problems, another highly effective approach is to design preconditioners, such as the well-known Hiptmair-Xu (HX) preconditioner \cite{hiptmair2007nodal}. Utilizing auxiliary space techniques, the HX preconditioner reduces the 3D problem to solving four Poisson problems, to which the multilevel algorithm can be applied. Building on this idea, we can also consider modifying some of the transfer operators related to the auxiliary spaces within HX preconditioners, employing operators associated with convection, such as flux operators ${\mathcal J}^{\rm curl}_{\varepsilon,\bm{\beta}}$ and some weighted interpolations. Designing and analyzing preconditioners along these lines is a topic of our ongoing research, and we hope to make progress in the future.


\bibliographystyle{siamplain}
\bibliography{SAFE_solver}

\end{document}